\documentclass[11pt]{article}
\usepackage{amssymb}
\usepackage{amsbsy}
\usepackage[latin1]{inputenc}
\usepackage{amsthm}
\usepackage[dvips]{graphicx}
\usepackage{graphicx} 
\usepackage{subfigure}
\usepackage{pst-eucl}
\usepackage[latin1]{inputenc}
\usepackage[english]{babel}
\usepackage{amsmath,amssymb,graphics,mathrsfs}
\usepackage{amsmath,amssymb,latexsym,amsfonts}
\usepackage{graphicx,color}
\usepackage[T1]{fontenc}
\usepackage[active]{srcltx}
\usepackage{multicol}
\usepackage[latin1]{inputenc}
\usepackage{pst-all}
\usepackage{enumerate}
\usepackage{pstricks}
\usepackage{pstricks-add}
\usepackage{setspace}
\usepackage{soul}
\usepackage{cancel}
\usepackage{nonfloat}
\usepackage[margin=10pt,font=footnotesize,labelfont=bf,labelsep=endash]{caption}
\usepackage[left=4cm,top=3cm,right=2.4cm,bottom=3.2cm]{geometry}
\usepackage[colorlinks=true,citecolor=red,linkcolor=blue,urlcolor=RubineRed,pdfpagetransition=Blinds,pdftoolbar=false,pdfmenubar=false]{hyperref}

\newtheorem{definition}{Definition}
\newtheorem{theorem}{Theorem}
\newtheorem{corol}{Corollary}
\newtheorem{lemma}{Lemma}
\newtheorem{remark}{Remark}

\begin{document}
\title{Optimal bilinear control problem related to a chemo-repulsion system in 2D domains}
\author{F. Guill\'en-Gonz\'alez$^1$, E. Mallea-Zepeda$^2$, M.A. Rodr\'iguez-Bellido$^3$}
\date{\small$^{1,3}$\it Dpto. Ecuaciones Diferenciales y An\'alisis Num\'erico and IMUS\, Universidad de Sevilla, Sevilla, Spain\\
\small$^2$\it Departamento de Matem\'atica, Universidad de Tarapac\'a, Arica, Chile}
\maketitle
\footnotetext{$^1$ E-mail: {\tt guillen@us.es}}
\footnotetext{$^2$E-mail:{\tt emallea@uta.cl}}
\footnotetext{$^3$E-mail: {\tt angeles@us.es}}
\date{}

\begin{abstract}
In this paper we study a bilinear optimal control problem associated to a chemo-repulsion model with linear production term in a bidimensional domain. The existence, uniqueness and regularity of strong solutions of this model are deduced, proving the existence of an global optimal solution.  Afterwards, we derive  first-order optimality conditions by using a Lagrange multipliers theorem. 
\medskip

\medskip


\end{abstract}

\section{Introduction}
In biology, the  chemotaxis phenomenon is understood as the movement of living organisms induced by 
the presence of certain chemical substances.
In 1970 Keller and Segel \cite{keller-segel} proposed a mathematical model that describes 
chemotactic aggregation of cellular slime molds which move preferentially
towards relatively high concentrations of a chemical substance secreted by the amoebae
themselves. Such phenomenon  is called chemoattraction with production. In contrast, if regions of high chemical concentration generate a repulsive effect on the organisms,
the phenomenon is called chemorepulsion.

We are interested in studying a chemorepulsion model  given by the 
following system of partial differential equations
\begin{equation}\label{system}
\left\{
\begin{array}{rcl}
\partial_tu-\Delta u&=&\nabla\cdot(u\nabla v)\quad \mbox{in $(0,T)\times\Omega\equiv Q$,}\\
\partial_tv-\Delta v+v&=&h(u)\quad \mbox{ in $(0,T) \times\Omega\equiv Q$,}\\
u(0,x)&=&u_0(x),\ v(0,x)=v_0(x)\quad \mbox{ in $\Omega$,}\\
\dfrac{\partial u}{\partial{\bf n}}&=&0,\ \dfrac{\partial v}{\partial{\bf n}}=0\quad \mbox{ on $(0,T)\times\partial\Omega$,}
\end{array}
\right.
\end{equation}
where $\Omega\subset\mathbb{R}^2$, is a bounded domain with smooth boundary $\partial\Omega$, ${\bf n}$ denotes the outward unit 
normal vector to $\partial\Omega$ and $(0,T)$ is a time interval. The unknowns are  cell density $u(t,x)\ge 0$ and chemical concentration $v(t,x)\ge 0$.
The function $h(u)$ represents the production term, which must be nonnegative when $u\ge0$.

System (\ref{system}), when the production term is linear, that is $h(u)=u$, was studied by Cieslak et al in \cite{cieslak}. The authors, based on
the abstract theory for quasilinear parabolic problems (see \cite{amann}), proved the global existence
and uniqueness of smooth classical solution in 2D domains, and global existence of weak solutions in spaces of  dimension 3 and 4.
Tao \cite{tao}, in a bounded convex domain $\Omega\subset\mathbb{R}^n$ ($n\ge 3$), studies system (\ref{system}) with $h(u)=u$ and a modification
in the density-dependent chemotactic sensitivity function, that is, the term $\nabla\cdot(u\nabla v)$ is changed by $\nabla\cdot(g(u)\nabla v)$, where
$$
g\in C^2([0,+\infty]),\ g(0)=0,\  0<g(u)\le C(u+1)^\alpha\  \mbox{ for all }\ u>0 
$$
with some $C>0$  and  $\alpha>0.$ The author prove that, under assumptions of initial data $0\not\equiv u_0\in C^0(\overline{\Omega})$ and $v_0\in C^1(\overline{\Omega})$
are nonnegative and that $\alpha<\frac{4}{n+2}$, there exists a unique global in time classical solution of (\ref{system}) and the corresponding solution $(u,v)$ converges to
$(\overline{u}_0,\overline{u}_0)$ as time goes to $+\infty$, where $\overline{u}_0:=\frac{1}{|\Omega|}\int_{\Omega}u_0$.

In this work we study a control problem subject to this chemorepulsion with linear production model in which a bilinear control acts injecting
or extracting chemical substance on a subdomain of control $\Omega_c\subset\Omega$.
Specifically, we consider $\Omega\subset\mathbb{R}^2$ be a 
bounded domain of class $C^2$, then we study a control problem associated  to the following system in $Q:=(0,T)\times\Omega$, 

\begin{equation}\label{eq1}
\left\{\begin{array}{rcl}
\partial_tu-\Delta u&=&\nabla\cdot(u\nabla v),\\
\partial_tv-\Delta v+v&=&u+f v,\\ 
\end{array}
\right.
\end{equation}
with initial conditions 
\begin{equation}\label{eq2}
u(0,x)=u_0(x)\ge0,\ v(0,x)=v_0(x)\ge0\ \mbox{ in }\Omega,
\end{equation}
and boundary conditions 
\begin{equation}\label{eq3}
\dfrac{\partial u}{\partial{\bf n}}=0,\quad \dfrac{\partial v}{\partial{\bf n}}=0\ \mbox{ on }(0,T)\times\partial\Omega.
\end{equation}
Here, the function $f$ 
denotes a bilinear control that acts on chemical concentration, which lies in a closed convex set $\mathcal{F}$. We observe that in the subdomains where $f\ge0$
we inject chemical  substance, and conversely  where $f\le 0$ we extract chemical substance. There is a wide collection of publications dealing
with optimal control of PDEs. See, for example, \cite{alekseev,casas_kunisch,karl_whachsmuth,kiem_rosch,kunisch_trautman,elva_elder,diego_elder,tachim,wang-1,zheng}
and the references therein.
In all previous publications, the control variable enters the state equation either on the right-hand side (distributed controls) or is part of the boundary conditions
(boundary controls). As far as we know, the literature related to optimal control problems with bilinear control is scarce, see
\cite{borzi_park,fister_mccarthy,kroner_vexler,vallejos_borzi}. The main difficulty is that the solution of the state equation depends nonlinearly on the control and state variables
(see  the second equation in (\ref{eq1})).

In the context of optimal control problems associated to  chemotaxis models, the literature is also scarce, see \cite{dearaujo,fister_mccarthy, rodriguez_rueda,ryu_yagi,ryu}.
In \cite{dearaujo} the authors study a distributed optimal control for a two-dimensional  model of cancer invasion. Using the 
Leray-Schauder fixed point theorem, they prove the existence of weak solutions of state system. Also, they prove the existence of optimal control and derive an optimality system.
The works \cite{fister_mccarthy} and \cite{ryu} delimit their study to a one-dimensional domain. In \cite{fister_mccarthy} 
two extreme problems  on a chemoattractant  model are analyzed; one involves harvesting the actual cells and the other depicts removing a proportion of the chemical substance. 
The control is bilinear (total) and acts on a portion of the cells or chemical substance. 
They prove the existence of optimal solutions and derive an optimality system. Also, they design a numerical scheme for the optimality system and present some examples.
In the problem studied in \cite{ryu}, the control acts on the boundary conditions for the chemical substance. The existence of optimal solutions is proved.
In  the recent work \cite{rodriguez_rueda}, the authors analyze a distributive optimal control problem where the state equations are given
by a stationary chemotaxis model coupled with the Navier-Stokes equations (chemotaxis-fluid system). They prove the existence
of an optimal solution. In addition, they derive an optimality system through a penalty method, because the relation control-state is multivalued. Finally,
in \cite{ryu_yagi}, on a 2D domain, the authors study a problem in which the control variable is  distributed, and acts on the equation for the chemical substance. 
They prove the existence of optimal solutions. Furthermore, using the fact that the state is differentiable with respect to the control, they derive an optimality system.
Other studies related to controllability for the nonstationary Keller-Segel system and nonstationary chemotaxis-fluid model can be consulted in
\cite{chaves_guerrero_1} and \cite{chaves_guerrero_2}, respectively.

The outline of this paper is as follows:
In Section \ref{sec:2}, we fix the notation, introduce the functional spaces to be used, give the definition of strong solution for
system (\ref{eq1})-(\ref{eq3}) and we state a parabolic regularity result that will be used throughout this work. In Section
\ref{sec:Existence}, we prove the existence (and uniqueness) of strong solution of (\ref{eq1})-(\ref{eq3}) using the Leray-Schauder fixed point theorem.
In Section \ref{sec:4}, we establish the optimal  control problem,  proving the existence of an optimal solution
and we obtain the first-order optimality conditions  
based on a Lagrange multipliers theorem  in Banach spaces.
Finally, we obtain  a regularity result  for Lagrange multipliers.

\section{Preliminaries}
\label{sec:2}
In order to establish the control problem, we will introduce some notations. We will use the Lebesgue space
$L^p(\Omega)$, $1\le p\le+\infty$, with norm denoted by $\|\cdot\|_{L^p}$. In particular, the $L^2(\Omega)$ norm and its 
inner product  will denoted by $\|\cdot\|$ and $(\cdot,\cdot)$, respectively. We consider the usual Sobolev spaces
$W^{m,p}(\Omega)=\{u\in L^p(\Omega)\,:\, \|\partial^\alpha u\|_{L^p}<+\infty,\ \forall |\alpha|\le m\}$, 
with norm denoted by $\|\cdot\|_{W^{m,p}}$. When $p=2$, we write $H^m(\Omega):=W^{m,2}(\Omega)$ and we denote the respective norm by
$\|\cdot\|_{H^m}$. Also, we use the space $W_{\bf n}^{m,p}(\Omega)=\{u\in W^{m,p}(\Omega)\,:\, \frac{\partial u}{\partial{\bf n}}=0\mbox{ on }\partial\Omega\}$ ($m>1+1/p$), with norm denoted
by $\|\cdot\|_{W^{m,p}_{\bf n}}$. If $X$ is a Banach space, we denote by $L^p(0,T;X)$ 
the space of valued functions in $X$ defined on the interval
$[0,T]$ that are integrable in the Bochner sense, and its norm will be denoted by $\|\cdot\|_{L^p(X)}$. For  simplicity we denote
$L^p(Q):=L^p(0,T;L^p(\Omega))$ 
 if $p\not= +\infty$ and its norm by $\|\cdot\|_{L^p(Q)}$. 

In the case $p=+\infty$, $L^{\infty}(Q)$ means $L^{\infty}((0,T)\times \Omega)$, and its norm is denoted by $\|\cdot\|_{L^{\infty}(Q)}$.
Also, we denote by $C([0,T];X)$ the space of continuous functions from $[0,T]$ into a Banach space
$X$, and its norm by $\|\cdot\|_{C(X)}$.
The  topological dual space of a Banach space $X$ will be denoted by $X'$,
and the duality for a pair $X$ and $X'$ by $\langle\cdot,\cdot\rangle_{X'}$ or simply
by $\langle\cdot,\cdot\rangle$ unless this leads to ambiguity. Moreover, the letters $C$, ${K}$, $C_1$, ${K}_1,...,$ are positive constants,
independent of state $(u,v)$ and control $f$,  but its value may change from line to line.

We are interested in the study of a control problem associated to strong solutions of  system (\ref{eq1})-(\ref{eq3}). In the following definition we give the concept  of strong solution of (\ref{eq1})-(\ref{eq3}).
\begin{definition}\label{strong}
Let $f\in L^4(Q)$, $u_0\in H^1(\Omega)$, $v_0\in W^{3/2,4}_{{\bf n}}(\Omega)$ with $u_0\ge 0$ and $v_0\ge 0$ a.e.~in $\Omega$, a pair $(u,v)$
is called strong solution of problem (\ref{eq1})-(\ref{eq3}) in $(0,T)$, if $u\ge0$ and $v\ge 0$ in $Q$,
\begin{eqnarray}
&&u\in\mathcal{Y}_u:=\{u\in L^\infty(0,T;H^1(\Omega))\cap L^2(0,T;H^2_{\bf n}(\Omega)),\ \partial_tu\in L^2(Q)\},\label{st-1}\\
&&v\in \mathcal{Y}_v:=\{v\in L^\infty(0,T;W^{3/2,4}_{\bf n}(\Omega))\cap L^4(0,T;W_{\bf n}^{2,4}(\Omega)),\ \partial_tv\in L^4(Q)\},\label{st-11}
\end{eqnarray}
the system (\ref{eq1}) hold pointwisely a.e.~$(t,x)\in Q$,
\begin{eqnarray}
\partial_tu-\Delta u&=&\nabla\cdot(u\nabla v),\label{st-2}\\
\partial_tv-\Delta v+v&=&u+fv,\label{st-3}
\end{eqnarray}
and the boundary and initial conditions (\ref{eq2}) and (\ref{eq3}) are satisfied, respectively.
\end{definition}

\begin{remark}
The problem (\ref{eq1})-(\ref{eq3}) is conservative in $u$. In fact, integrating  (\ref{eq1})$_1$ in $\Omega$ we have
\begin{equation}\label{comp_u}
\frac{d}{dt}\left(\int_\Omega u\right)=0,\ \mbox{ i.e. }\ \int_\Omega u(t)=\int_\Omega u_0:=m_0,\ \forall t>0. 
\end{equation}
Also, integrating  (\ref{eq1})$_2$ in $\Omega$ we deduce 
\begin{equation}\label{comp_v}
\frac{d}{dt}\left(\int_\Omega v\right)+\int_\Omega v= m_0+\int_\Omega fv.
\end{equation}
\end{remark}
We define the space ${\widehat{W}^{2-2/p,p}}(\Omega)$ as follows
\begin{equation}\label{besov-4}
\widehat{W}^{2-2/p,p}(\Omega)=
\left\{\begin{array}{rcl}
W^{2-2/p,p}(\Omega)&\mbox{if}& p<3,\\
W^{2-2/p,p}_{\bf n}(\Omega)&\mbox{if}&p>3.
\end{array}
\right.
\end{equation}

In order to study the  existence of solution of system (\ref{eq1})-(\ref{eq3}), we will use the following  regularity result for the heat equation
(see  \cite{feireisl}, p. 344).

\begin{lemma}\label{feireisl}
For

$\Omega \in \mathcal{C}^2$, let
$1<p<+\infty$ 
 ($p\not=3$)
and  $g\in L^p(Q)$, $u_0\in\widehat{W}^{2-2/p,p}(\Omega)$. Then the problem
\begin{equation*}
\left\{
\begin{array}{rcl}
\partial_tu-\Delta u&=&g \mbox{ in }Q,\\
u(0,x)&=&u_0(x)\mbox{ in }\Omega,\\
\dfrac{\partial u}{\partial{\bf n}}&=&0\mbox{ on }(0,T)\times\partial\Omega,
\end{array}
\right.
\end{equation*}
admits a unique solution $u$ in the class
\begin{equation*}
u\in C([0,T];\widehat{W}^{2-2/p,p}(\Omega))\cap L^p(0,T;W^{2,p}(\Omega)),\ \partial_tu\in L^p(Q).
\end{equation*}
Moreover, there exists a positive constant $C=C(p,\Omega,T)$ such that
\begin{equation}\label{des-regularity}
\|u\|_{C(\widehat{W}^{2-2/p,p})}+\|\partial_tu\|_{L^p(Q)}+\|u\|_{L^p(W^{2,p})}
\le C(\|g\|_{L^p(Q)}+\|u_0\|_{\widehat{W}^{2-2/p,p}}).
\end{equation}

In particular, the equation $\partial_tu-\Delta u=g$ is pointwisely satisfied  a.e.~in $Q$. 

\end{lemma}

\begin{remark}
In the case of $p=3$, one concludes that $u\in C([0,T];X_{3,3})\cap L^3(0,T;W^{2,3}(\Omega))$, 
$\partial_tu\in L^3(Q)$, for a certain space $X_{3,3}$ (see \cite[Theorem 10.22]{feireisl}) whose description  is not evident in terms of $\widehat{W}^{2-2/p,p}(\Omega)$ or another Sobolev space.
\end{remark}

Thorough this paper, we will use the following equivalent norms in $H^1(\Omega)$ and $H^2(\Omega)$,   respectively (see \cite{necas} for details):
\begin{eqnarray}
\|u\|^2_{H^1}&\simeq&\|\nabla u\|^2+\left(\int_\Omega u\right)^2,\quad\forall\, u\in H^1(\Omega),\label{equi}\\
\|u\|^2_{H^2}&\simeq&  \Vert \Delta u \Vert^2+ \left(
\displaystyle\int_{\Omega} u 
\right)^2, \quad\forall u\in H^2_{\bf n}(\Omega),\label{norma-1}
\end{eqnarray}
and the classical interpolation  inequality in $2D$ domains
\begin{equation}\label{interpol}
\|u\|_{L^4}\le C\|u\|^{1/2}\|u\|^{1/2}_{H^1},\quad \forall\, u\in H^1(\Omega).
\end{equation}

\section{Existence and Uniqueness of Strong Solution of System  (\ref{eq1})-(\ref{eq3})}
\label{sec:Existence}

In this section we will prove the existence (and uniqueness) of solution of (\ref{eq1})-(\ref{eq3}) using the Leray-Schauder fixed point theorem.
Specifically we will prove the following result:

\begin{theorem}\label{strong_solution}
Let $u_0\in H^{1}(\Omega)$, $v_0\in  W^{3/2,4}_{{\bf n}}(\Omega)$ with $u_0\ge 0$ and $v_0\ge 0$ in $\Omega$, and $f\in L^{4}(Q)$. 
There exists a unique strong solution $(u,v)$ of system (\ref{eq1})-(\ref{eq3}) in sense of Definition \ref{strong}. Moreover, there exists a positive constant
$$
\mathcal{K}_1:=\mathcal{K}_1(m_0,T,\|u_0\|_{H^1},\|v_0\|_{W^{3/2,4}_{\bf n}},\|f\|_{L^4(Q)}),
$$ 
such that
\begin{equation}\label{bound_solution}
\|\partial_tu,\partial_tv\|_{L^2(Q)\times L^4(Q)}+\|u,v\|_{C(H^1\times W^{3/2,4}_{\bf n})}+\|u\|_{L^2(H^2)}+\|v\|_{L^4(W^{2,4})}\le \mathcal{K}_1.
\end{equation}
\end{theorem}

\subsection{Existence}\label{existence}

Let us  introduce the ``weak'' spaces
\begin{equation}\label{pf}
\mathcal{X}_u:= L^\infty(0,T;L^2(\Omega))\cap L^{8/3}(0,T;W^{1,8/3}(\Omega))\quad \mbox{ and }\quad
\mathcal{X}_v:= C^0([0,T];C(\overline{\Omega}))
\end{equation}

We define the operator
$R:\mathcal{X}_u\times \mathcal{X}_v\rightarrow \mathcal{Y}_u\times\mathcal{Y}_v\hookrightarrow \mathcal{X}_u\times \mathcal{X}_v$
by $R(\bar{u},\bar{v})=(u,v)$ the solution of the decoupled  linear problem
\begin{equation}\label{pf-1}
\left\{
\begin{array}{rcl}
\partial_tu-\Delta u&=&\nabla\cdot(\bar{u}_+\nabla v),\\
\partial_tv-\Delta v+v&=&\bar{u}_++f\bar{v}_+,\\
u(0)&=&u_0,\ v(0)=v_0,\\
\dfrac{\partial u}{\partial{\bf n}}&=&0,\ \dfrac{\partial v}{\partial{\bf n}}=0,
\end{array}
\right.
\end{equation}
where $\bar{u}_+:=\max\{\bar{u},0\}\ge0$, $\bar{v}_+:=\max\{\bar{v},0\}\ge0$. In fact, first we find $v$ and after $u$. 

In the following lemmas we will prove the hypotheses  of Leray-Schauder fixed point theorem.
\begin{lemma}\label{compact}
The operator $R:\mathcal{X}_u\times\mathcal{X}_v\rightarrow \mathcal{X}_u\times\mathcal{X}_v$ is well defined and compact.
\end{lemma}
\begin{proof}
Let $(\bar{u},\bar{v})\in \mathcal{X}_u\times\mathcal{X}_v$. Since $\mathcal{X}_u\hookrightarrow L^4(Q)$, 
$\mathcal{X}_v\subset L^\infty(Q)$ and $f\in L^4(Q)$, then $\bar{u}_++f\bar{v}_+\in L^4(Q)$. By applying Lemma \ref{feireisl} (for $p=4$), there exists a unique solution
$v\in\mathcal{Y}_v$ of (\ref{pf-1})$_2$ such that
\begin{eqnarray}
\|v\|_{L^4(W^{2,4})}+\|\partial_tv\|_{L^4(Q)}+\|v\|_{C(W^{3/2,4}_{\bf n})}&\le& C(\|\bar{u}\|_{L^4(Q)}+\|\bar{v}\|_{L^\infty(Q)}\|f\|_{L^4(Q)}+\|v_0\|_{W^{3/2,4}_{\bf n}})\nonumber\\
&\le&C(\|v_0\|_{W^{3/2,4}_{\bf n}},\|f\|_{L^4(Q)}).\label{com_2}
\end{eqnarray}
Now, using the fact that $v\in\mathcal{Y}_v$, in particular we have $\nabla v\in L^\infty(0,T;L^4(\Omega))\cap L^4(0,T;W^{1,4}(\Omega))\hookrightarrow L^8(Q)$, and taking into account that
$\nabla\bar{u}_+\in L^{8/3}(Q)$, $\Delta v\in L^4(Q)$, $\bar{u}_+\in L^\infty(0,T;L^2(\Omega))\cap L^2(0,T;H^1(\Omega))\hookrightarrow L^4(Q)$ 
we have $\nabla\cdot(\bar{u}_+\nabla v)=\bar{u}_+\Delta v+\nabla\bar{u}_+\cdot\nabla v\in L^{2}(Q).$ Thus, again by Lemma \ref{feireisl} (for $p=2$), we conclude that there exist  a unique $u\in\mathcal{Y}_u$ solution of
(\ref{pf-1})$_1$ such that
\begin{eqnarray}
\| u\|_{L^2(H^2)}+\|\partial_tu\|_{L^2(Q)}+\|u\|_{C(H^1)}
&\le& C(\|\bar{u}\|_{L^4(Q)}\|\Delta v\|_{L^4(Q)}+\|\nabla\bar{u}\|_{L^{8/3}(Q)}\|\nabla v\|_{L^8(Q)}+\|u_0\|_{H^1})\nonumber\\
&\le& C(\|u_0\|_{H^1},\|v_0\|_{W^{3/2,4}_{\bf n}},\|f\|_{L^4(Q)}).\label{com_3}
\end{eqnarray}
In the last inequality of \eqref{com_3}, estimate \eqref{com_2} has been used. 
Therefore,  $R$ is well defined from $\mathcal{X}_u\times\mathcal{X}_v$ to $\mathcal{Y}_u\times\mathcal{Y}_v$. 

Moreover, the compactness of $R$ is consequence of estimates (\ref{com_2}) and (\ref{com_3}), and the compact embedding 
$\mathcal{Y}_u\times\mathcal{Y}_v\hookrightarrow\mathcal{X}_u\times\mathcal{X}_v$.

Indeed, let  $(\bar{u},\bar{v})$ be in a bounded set of $\mathcal{X}_u\times\mathcal{X}_v$ and consider $(u,v)=R(\bar{u},\bar{v})$. From (\ref{com_3}),  $u$ is bounded in $L^{\infty}(0,T;H^1(\Omega))$ and $\partial_t u$ is bounded in $L^2(Q)$. Because of $H^1(\Omega) \hookrightarrow L^2(\Omega)$ in a compact way, from \cite[Th\'eor\`eme 5.1]{simon} one can  deduce that $\mathcal{Y}_u$ is compactly embedded in $C^0([0,T];L^2(\Omega))$. 
Moreover,  interpolating between $L^{\infty}(0,T;H^1(\Omega))$ and $L^{2}(0,T;H^2(\Omega))$,  $u$ is bounded in $L^{8/3}(0,T;H^{7/4}(\Omega))$ (see \cite{feireisl} or \cite{triebel}). 
Using that $H^{7/4}(\Omega) \hookrightarrow W^{1,p}(\Omega)$, $p<8$, in a compact way, and that $\partial_t u$ is bounded in $L^2(Q)$, again from \cite[Th\'eor\`eme 5.1]{simon} one has in particular  that
$\mathcal{Y}_u$ is compactly embedded in $L^{8/3}(0,T;W^{1,8/3}(\Omega))$.
Similarly, starting from (\ref{com_2}),  $v$ is bounded in $L^{\infty}(0,T;W_{\bf n}^{3/2,4}(\Omega))$ and $\partial_t v$ is bounded in $L^4(Q)$. 
Taking into account that $W_{\bf n}^{3/2,4}(\Omega)\hookrightarrow C^0(\bar{\Omega})$ in a compact way, one can  deduce that $\mathcal{Y}_v$ is compactly embedded in $C^0([0,T];C^0(\overline{\Omega}))=C^0([0,T]\times \overline{\Omega})$.

\end{proof}

\begin{lemma}\label{cot-1}
The set
\begin{equation}\label{cot-2}
T_\alpha=\{(u,v)\in\mathcal{Y}_u\times\mathcal{Y}_v\,:\, (u,v)=\alpha R(u,v)\mbox{ for some }\alpha\in[0,1]\}
\end{equation}
is bounded in $\mathcal{X}_u\times\mathcal{X}_v$
(independently of $\alpha \in [0,1]$).  In fact, $T_\alpha$ is also bounded in $\mathcal{Y}_u\times\mathcal{Y}_v$, i.e.~there exists 
\begin{equation}\label{cot-3}
M=M(m_0,T,\|u_0\|_{H^1},\|v_0\|_{W^{3/2,4}_{\bf n}},\|f\|_{L^4(Q)})>0,
\end{equation}
with $M$ independent of $\alpha$, such that 
all pairs of functions $(u,v)\in T_\alpha$ for $\alpha\in[0,1]$ satisfy
$$
\|(u,v)\|_{L^2(H^2)\times L^4(W^{2,4})}+\|(\partial_tu,\partial_tv)\|_{L^2(Q)\times L^4(Q)}+\|(u,v)\|_{C(H^1)\times C(W^{3/2,4}_{\bf n})} \le M.
$$

\end{lemma}
\begin{proof}
 Let $(u,v)\in T_\alpha$ for $\alpha\in (0,1]$ (the case  $\alpha=0$  is trivial). Then, 

owing to Lemma 3.2, $(u,v) \in \mathcal{Y}_u \times \mathcal{Y}_v$ and satisfies pointwisely a.e.~in $Q$ the following problem:
\begin{equation}\label{pf-7}
\left\{
\begin{array}{rcl}
\partial_tu-\Delta u&=&\nabla\cdot({u}_+\nabla v),\\
\partial_tv-\Delta v+v&=&\alpha{u}_++\alpha f{v}_+
\end{array}
\right.
\end{equation}
endowed with the corresponding initial and boundary conditions. Therefore, it suffices to look a bound of $(u,v)$ in $\mathcal{Y}_u \times \mathcal{Y}_v$ independent of $\alpha$. 
This bound is carried out into five steps:

\vspace{0.2cm}
\noindent{\underline{Step 1:} \ $u,v\ge 0$ and $\displaystyle\int_\Omega u(t)=m_0$.
\vspace{0.2cm}

By testing (\ref{pf-7})$_1$ by $u_-:=\min\{u,0\}\le0$,
and considering that $u_-=0$ if $u\ge 0$, $\nabla u_-=\nabla u$ if $u\le 0$, and $\nabla u_-=0$ if $u>0$, we have
$$
\frac12\frac{d}{dt}\|u_-\|^2+\|\nabla u_-\|^2=-(u_+\nabla v,\nabla u_-)=0,
$$
thus $u_-\equiv 0$ and, consequently, $u\ge0$.  Similarly, testing  (\ref{pf-7})$_2$ by $v_-:=\min\{v,0\}\le0$ 
we obtain
$$
\frac12\frac{d}{dt}\|v_-\|^2+\|\nabla v_-\|^2+\|v_-\|^2=\alpha(u_+,v_-)+\alpha(fv_+,v_-)\le 0,
$$
which implies $v_-\equiv0$, then $v\ge0$. Therefore $(u_+,v_+)=(u,v)$. Finally, integrating (\ref{pf-7})$_1$ in $\Omega$ and using
(\ref{comp_u}) we obtain $\displaystyle\int_\Omega u(t)=m_0$.

\vspace{0.2cm}
\noindent\underline{Step 2:} $v$ is bounded in $L^\infty(0,T;H^1(\Omega))\cap L^2(0,T;H^2(\Omega))$.}
\vspace{0.2cm}

We observe that, thanks to the positivity  of $u$, we have $0\le \ln(u+1)\le u$. Then 
\begin{equation}\label{cot-4}
\int_\Omega|\ln(u+1)|^2\le \int_\Omega|u|^2.
\end{equation}
We also note that
\begin{equation}\label{cot-5}
\int_\Omega|\nabla\ln(u+1)|^2=\int_\Omega\left|\frac{\nabla u}{u+1}\right|^2\le \int_\Omega|\nabla u|^2.
\end{equation}
%
%
Taking into account that $u\in L^2(0,T;H^1(\Omega))$, from (\ref{cot-4}) and (\ref{cot-5}) we deduce that $\ln(u+1)\in L^2(0,T;H^1(\Omega))$.
Then, 
since $u\in \mathcal{Y}_u$,
 testing (\ref{pf-7})$_1$ by $\alpha\ln(u+1)\in L^2(0,T;H^1(\Omega))$ and (\ref{pf-7})$_2$ by $-\Delta v \in L^4(Q)$, and integrating by parts, we have
%
%
%
\begin{eqnarray}
&&\frac{d}{dt}\left[ \alpha\int_\Omega(u+1)\ln(u+1)+\frac12\|\nabla v\|^2\right]+4\alpha\|\nabla\sqrt{u+1}\|^2+\|\Delta v\|^2+\|\nabla v\|^2\nonumber\\
&&=-\alpha\int_\Omega\frac{u}{u+1}\nabla v\cdot\nabla u+\alpha\int_\Omega\nabla u\cdot\nabla v-\alpha\int_\Omega fv\Delta v\nonumber\\
&&=\alpha\int_\Omega\frac{1}{u+1}\nabla u\cdot\nabla v-\alpha\int_\Omega fv\Delta v.\label{pf-9}
\end{eqnarray}
Applying the H\"older and Young inequalities we obtain
\begin{equation} \label{pf-10}
\alpha\int_\Omega\frac{1}{u+1}\nabla u\cdot\nabla v
\le\frac{\alpha}{2}\int_\Omega\frac{|\nabla u|^2}{u+1}+\frac{\alpha}{2}\int_\Omega\frac{|\nabla v|^2}{u+1}
\le2\alpha\|\nabla\sqrt{u+1}\|^2+\frac{\alpha}{2}\|\nabla v\|^2,
\end{equation}
\begin{equation} \label{pf-11}
-\alpha\int_\Omega fv\Delta v
\le\alpha\|f\|_{L^4}\|v\|_{L^4}\|\Delta v\|\le\delta\|v\|^2_{H^2}+\alpha^2C_\delta\|f\|^2_{L^4}\|v\|^2_{H^1}.
\end{equation}
Moreover, integrating (\ref{pf-7})$_2$ in $\Omega$ and using (\ref{comp_u}) and (\ref{comp_v}), we have
$$
\frac{d}{dt}\left(\int_\Omega v\right)+\int_\Omega v=\alpha\, m_0+\alpha\int_\Omega fv.
$$
Multiplying this equation by $\displaystyle\left(\int_\Omega v\right)$ and using the H\"older and Young inequalities we obtain
\begin{eqnarray}
\frac12\frac{d}{dt}\left(\int_\Omega v\right)^2+\left(\int_\Omega v\right)^2&=&\alpha\, m_0\left(\int_\Omega v\right)+\alpha\left(\int_\Omega fv\right)\left(\int_\Omega v\right)\nonumber\\
&\le&\frac12\left(\int_\Omega v\right)^2+C\alpha^2m_0^2+C\alpha^2\| f\|^2\|v\|^2\label{pf-12}.
\end{eqnarray}

Replacing (\ref{pf-10})-(\ref{pf-12}) in (\ref{pf-9}), and taking into account that $\alpha\le1$, we  obtain
\begin{equation}\label{pf-13}
\frac{d}{dt}\left(\alpha\int_\Omega(u+1)\ln(u+1)+\frac12\|v\|^2_{H^1}\right)+2\alpha\|\nabla\sqrt{u+1}\|^2+C_1\|v\|^2_{H^2}
\le C(m_0^2+\|f\|^2_{L^4}\|v\|^2_{H^1}),
\end{equation}
where the constants $C,C_1$ are independent of $\alpha$. 
From (\ref{pf-13}) and Gronwall lemma we have 
\begin{eqnarray}
\|v\|^2_{L^\infty(H^1)}
\le \exp(\|f\|^2_{L^2(L^4)})\left(\|u_0\|^2+\|v_0\|^2_{H^1}+Cm_0^2T\right)
:= K_0(m_0,T,\|u_0\|,\|v_0\|_{H^1},\|f\|_{L^2(L^4)}).\label{e-31}
\end{eqnarray}
Now, integrating  (\ref{pf-13}) in $(0,T)$ and using (\ref{e-31}) we obtain 
\begin{eqnarray}
\|v\|^2_{L^2(H^2)}&\le&
C\left(\|u_0\|^2+\|v_0\|^2_{H^1}+m_0^2T+\|v\|^2_{L^\infty(H^1)}\|f\|^2_{L^2(L^4)}\right)\nonumber\\
&:=&K_1(m_0,T,\|u_0\|, \|v_0\|_{H^1},\|f\|_{L^2(L^4)}).\label{e-33}
\end{eqnarray}
Therefore, from (\ref{e-31}) and  (\ref{e-33}) we conclude that
$v$ is bounded in $L^{\infty}(0,T;H^1(\Omega))\cap L^2(0,T;H^2(\Omega))$.

\vspace{0.2cm}
\noindent{\underline{Step 3:} $u$ is bounded in $L^\infty(0,T;L^2(\Omega))\cap L^2(0,T;H^1(\Omega))$.}
\vspace{0.2cm}

Testing (\ref{pf-7})$_1$ by $u$, applying the H\"older and Young inequalities, and using (\ref{interpol}), we obtain
\begin{eqnarray*}
\frac12\frac{d}{dt} \Vert u \Vert^2 + \Vert \nabla u \Vert^2 
&=&-(u \, \nabla v, \nabla u)\le \|u\|_{L^4}\|\nabla v\|_{L^4}\|\nabla u\|\le C\|u\|^{1/2}\|\nabla v\|_{L^4}\|u\|^{3/2}_{H^1}\\
&\le&C\|u\|^2\|\nabla v\|^4_{L^4}+\frac{1}{2}\|u\|^2_{H^1}.
\end{eqnarray*}
Thus,   taking into account that $m_0^2=\left(\displaystyle\int_\Omega u(t)\right)^2$ and the  equivalent norm of $H^1(\Omega)$ given in \eqref{equi}, we have
\begin{equation}\label{e-7}
\frac{d}{dt} \Vert u \Vert^2 + \Vert  u \Vert_{H^1}^2 \le 
C \, \Vert  \nabla v \Vert^4_{L^4}\Vert u \Vert^2 +2 m_0^2.
\end{equation}
On the other hand, using (\ref{interpol}), jointly (\ref{e-31}) and (\ref{e-33}),
$$
\|\nabla v\|^4_{L^4(Q)}\le CK_0K_1.
$$
Therefore, we can  apply the Gronwall lemma in (\ref{e-7}), obtaining
\begin{equation}\label{e-70}
\|u\|^2_{L^\infty(L^2)}\le \exp(CK_0K_1)(\|u_0\|^2+2m_0^2):=K_2(m_0,T,\|u_0\|,\|v_0\|_{H^1},\|f\|_{L^2(L^4)}).
\end{equation}
Also, integrating (\ref{e-7}) in $(0,T)$ we have
\begin{eqnarray}
\|u\|^2_{L^2(H^1)}&\le&\|u_0\|^2+2m_0^2T+CK_0K_1 \|u\|^2_{L^\infty(L^2)}
\le\|u_0\|^2+2m_0^2T+CK_0K_1K_2 \nonumber\\
&:=&K_3(m_0,T,\|u_0\|,\|v_0\|_{H^1},\|f\|_{L^2(L^4)}).\label{e-700}
\end{eqnarray}

Therefore, from (\ref{e-70}) and (\ref{e-700}) we deduce that $u$ is bounded in $L^\infty(0,T;L^2(\Omega))\cap L^2(0,T;H^1(\Omega))$.

\vspace{0.2cm}
\noindent{\underline{Step 4:} $v$ is bounded in $\mathcal{Y}_v$.}
\vspace{0.2cm}

Taking into account that $f\in L^4(Q)$ and $v\in L^\infty(0,T;H^1(\Omega))$, in particular $\alpha fv\in L^{7/2}(Q)$. 
Then using Lemma \ref{feireisl} (for $p=\frac72$) in (\ref{pf-7})$_2$ we conclude that  $v$ satisfies the following inequality 
\begin{eqnarray}\label{cot-8}
\|v\|_{L^{7/2}(W^{2,7/2})}+\|\partial_tv\|_{L^{7/2}(Q)}+\|v\|_{C(W^{10/7,7/2}_{\bf n})}
&\le&  C(\alpha\|u+fv\|_{L^{7/2}(Q)}+\|v_0\|_{W^{10/7,7/2}_{\bf n}})\nonumber\\
&\le& C(\|u\|_{L^4(Q)}+\|f\|_{L^{4}(Q)}\|v\|_{L^{28}(Q)}+\|v_0\|_{W^{10/7,7/2}_{\bf n}})\label{e-700a}.
\end{eqnarray}
From \eqref{e-31}, $\|v\|_{L^{28}(Q)}\le K_0$.
Using (\ref{interpol}) and taking into account (\ref{e-70}) and (\ref{e-700})  we have
\begin{equation}\label{estim_2}
\|u\|^4_{L^4(Q)}\le C\|u\|^2_{L^\infty(L^2)}\|u\|^2_{L^2(H^1)}\le CK_2K_3.
\end{equation}
Therefore, from \eqref{e-700a} one has $\|v\|_{C(W^{10/7,7/2}_{\bf n})}$ is bounded (independently of $\alpha$). 
In particular, by Sobolev embeddings, we obtain $\|v\|_{L^\infty(Q)}$ is also bounded. 

Then, from (\ref{estim_2}) and using again Lemma \ref{feireisl} (for $p=4$), we obtain that $v$ satisfies the estimate 
\begin{eqnarray}\label{estim_1}
\| v\|_{L^4(W^{2,4})}+\|\partial_t v\|_{L^4(Q)}+\|v\|_{C({W}^{3/2,4}_{\bf n})}
&\le&C(\alpha\|{u}+fv\|_{L^4(Q)}+\|v_0\|_{{W}^{3/2,4}_{\bf n}})\nonumber\\
&\le& C(\|{u}\|_{L^4(Q)}+\|v\|_{L^\infty(Q)}\|f\|_{L^4(Q)}+\|v_0\|_{{W}^{3/2,4}_{\bf n}})\nonumber\\
&\le&K_4(m_0,T,\|u_0\|,\|v_0\|_{W^{3/2,4}_{\bf n}},\|f\|_{L^4(Q)}).
\end{eqnarray}
Therefore $v$ is bounded in $\mathcal{Y}_v$.

\vspace{0.2cm}
\noindent{\underline{Step 5:} $u$ is bounded in $\mathcal{Y}_u$.}
\vspace{0.2cm}

Testing (\ref{pf-7})$_1$ by $-\Delta u\in L^2(Q)$ we have 
\begin{equation}\label{fij_1}
\frac12\frac{d}{dt}\|\nabla u\|^2+\|\Delta u\|^2=-(\nabla\cdot(u\nabla v),\Delta u)=-(u\Delta v+\nabla u\cdot\nabla v,\Delta u).
\end{equation}
By the  H\"older and Young inequalities, and using interpolation inequality (\ref{interpol}), we obtain 
\begin{eqnarray}
-(u\Delta v+\nabla u\cdot\nabla v,\Delta u)&\le&(\|u\|_{L^4}\|\Delta v\|_{L^4}+\|\nabla u\|_{L^4}\|\nabla v\|_{L^4})\|\Delta u\|\nonumber\\
&\le&\delta\|\Delta u\|^2+C_\delta\|u\|^2_{L^4}\|\Delta v\|^2_{L^4}+C\|\nabla u\|^{1/2}\|\nabla v\|_{L^4}\|u\|^{3/2}_{H^2}\nonumber\\
&\le&\delta\|\Delta u\|^2+C_\delta\|u\|^2_{L^4}\|\Delta v\|^2_{L^4}+C_\delta\|\nabla u\|^2\|\nabla v\|^4_{L^4}+\delta\|u\|^{2}_{H^2}.\label{fij_2}
\end{eqnarray}
Replacing  (\ref{fij_2}) in (\ref{fij_1}), choosing $\delta$ small enough to absorb the 
$\|\Delta u\|^2$ and $\| u\|_{H^2}^2$ terms, and taking into account that $\left(\displaystyle\int_\Omega u(t)\right)^2=m_0^2$ and \eqref{equi} and \eqref{norma-1}, we have
\begin{equation}\label{fij_3}
\frac{d}{dt}\|u\|^2_{H^1}+C\|u\|^2_{H^2}\le C\|u\|^2_{L^4}\|\Delta v\|^2_{L^4}+C\|\nabla u\|^2\|\nabla v\|^4_{L^4}+Cm_0^2.
\end{equation}
Then, from \eqref{estim_2} and  (\ref{estim_1}), and applying Gronwall lemma to (\ref{fij_3}), we deduce
\begin{equation}\label{fij_4}
\|u\|^2_{L^\infty(H^1)}\le K_5(m_0,T,\|u_0\|_{H^1}, \|v_0\|_{W^{3/2,4}_{\bf n}},\|f\|_{L^4(Q)}).
\end{equation}
Finally, integrating (\ref{fij_3}) in $(0,T)$ we obtain
\begin{equation}\label{fij_5}
\|u\|^2_{L^2(H^2)} \le K_6(m_0,T,\|u_0\|_{H^1}, \|v_0\|_{W^{3/2,4}_{\bf n}},\|f\|_{L^4(Q)}).
\end{equation}
Then, from (\ref{pf-7})$_1$, (\ref{estim_1}), (\ref{fij_4}) and (\ref{fij_5}) we have
\begin{eqnarray}\label{fij_6}
\|\partial_tu\|_{L^2(Q)}&=&\|\Delta u+u\Delta v+\nabla u\cdot\nabla v\|_{L^2(Q)}\nonumber\\
&\le&\|\Delta u\|_{L^2(Q)}+\|u\|_{L^4(Q)}\|\Delta v\|_{L^4(Q)}+\|\nabla u\|_{L^4(Q)}\|\nabla v\|_{L^4(Q)}\nonumber\\
&\le&K_7(m_0,T,\|u_0\|_{H^1}, \|v_0\|_{W^{3/2,4}_{\bf n}},\|f\|_{L^4(Q)}),
\end{eqnarray}
which implies that $u$ is bounded in $\mathcal{Y}_u$.

Finally, from (\ref{estim_1}) and  (\ref{fij_4})-(\ref{fij_6}) we conclude that the elements of $T_\alpha$ are bounded in $\mathcal{Y}_u\times\mathcal{Y}_v$ for $\alpha\in (0,1]$.
 The radius $M$ in (\ref{cot-3}) follows from (\ref{estim_1}) and  (\ref{fij_4})-(\ref{fij_6}).
\end{proof}

\begin{lemma}\label{conti-operator}
The operator 
$R:\mathcal{X}_u\times \mathcal{X}_v\rightarrow \mathcal{X}_u\times \mathcal{X}_v$, defined in
(\ref{pf-1}), is continuous.
\end{lemma}
\begin{proof}
Let $\{(\bar{u}_m,\bar{v}_m)\}_{m\in\mathbb{N}}\subset \mathcal{X}_u\times \mathcal{X}_v$ be a sequence such that
\begin{equation}\label{cont_1}
(\bar{u}_m,\bar{v}_m)\rightarrow(\bar{u},\bar{v})\mbox{ in }\mathcal{X}_u\times \mathcal{X}_v.
\end{equation}
In particular, $\{(\bar{u}_m,\bar{v}_m)\}_{m\in\mathbb{N}}$ is bounded in $\mathcal{X}_u\times \mathcal{X}_v$, thus, from (\ref{com_2}) and (\ref{com_3}) we deduce that the sequence 
$\{(u_m,v_m):=R(\bar{u}_m,\bar{v}_m)\}_{m\in\mathbb{N}}$ is bounded in  $\mathcal{Y}_u\times \mathcal{Y}_v$. Then,
from the compactness of $\mathcal{Y}_u\times\mathcal{Y}_v$ in $\mathcal{X}_u\times \mathcal{X}_v$ (see the proof of Lemma \ref{compact}),
 there exists a subsequence
of $\{R(\bar{u}_m,\bar{v}_m)\}_{m\in\mathbb{N}}$, still denoted by  $\{R(\bar{u}_m,\bar{v}_m)\}_{m\in\mathbb{N}}$, and an element
$(\tilde{u},\tilde{v})\in\mathcal{Y}_u\times\mathcal{Y}_v$ such that
\begin{equation}\label{cont_2}
R(\bar{u}_m,\bar{v}_m)\rightarrow (\tilde{u},\tilde{v})\mbox{ weak in }\mathcal{Y}_u\times\mathcal{Y}_v\mbox{ and strong in }\mathcal{X}_u\times \mathcal{X}_v.
\end{equation}

From (\ref{cont_1}) and  (\ref{cont_2}) we can take  the limit in (\ref{pf-1}), when $m$ goes to $+\infty$, with
$(u,v)=R(\bar{u}_m,\bar{v}_m)$ and $(\bar{u},\bar{v})=(\bar{u}_m,\bar{v}_m)$, which implies  that
$R(\bar{u},\bar{v})=(\tilde{u},\tilde{v})$. Then, by the uniqueness of limit the whole sequence $\{R(\bar{u}_m,\bar{v}_m)\}_{m\in\mathbb{N}}$ converges to 
$R(\bar{u},\bar{v})$ strongly in $\mathcal{X}_u\times\mathcal{X}_v$. Thus, operator $R$ is continuous from $\mathcal{X}_u\times\mathcal{X}_v$ into itself. 
\end{proof}
Consequently, from Lemmas  {\ref{compact}}, \ref{cot-1} and \ref{conti-operator}, it follows that the operator $R$ and the set $T_\alpha$ satisfy
the conditions of the Leray-Schauder fixed point theorem. Thus, we conclude that the map $R(\bar{u},\bar{v})$ has a fixed point, 
$R(u,v)=(u,v)$, which is a solution to system (\ref{eq1})-(\ref{eq3}).

Finally, we observe that estimate (\ref{bound_solution})  follows from (\ref{estim_1}) and (\ref{fij_4})-(\ref{fij_6}).
\subsection{Uniqueness}
Let $(u_1,v_1),\, (u_2,v_2)\in\mathcal{Y}_u\times\mathcal{Y}_v$ two solutions of system (\ref{eq1})-(\ref{eq3}). Subtracting  equations (\ref{eq1})-(\ref{eq3}) for
$(u_1,v_1)$ and $(u_2,v_2)$, and  denoting
$u:=u_1-u_2$ and $v:=v_1-v_2$, we obtain the following system
\begin{equation}\label{uni-1}
\left\{
\begin{array}{rcl}
\partial_tu-\Delta u&=&\nabla\cdot(u_1\nabla v+u\nabla v_2)\ \mbox{ in }Q,\\
\partial_tv-\Delta v+v&=&u+fv\ \mbox{ in }Q,\\
u(0,x)&=&0,\ v(0,x)=0\ \mbox{ in }\Omega,\\
\dfrac{\partial u}{\partial {\bf n}}&=&0,\ \dfrac{\partial v}{\partial{\bf n}}=0\ \mbox{ on }(0,T)\times\partial\Omega.
\end{array}
\right.
\end{equation}
Testing (\ref{uni-1})$_1$ by $u$ and (\ref{uni-1})$_2$ by $-\Delta v$  we have
\begin{eqnarray}
\frac{d}{dt}\left(\|u\|^2+\frac12\|\nabla v\|^2\right)+\|\nabla u\|^2+\|\Delta v\|^2+\|\nabla v\|^2
=-(u_1\nabla v,\nabla u)-(u\nabla v_2,\nabla u)+(u,-\Delta v)+(fv,-\Delta v).\label{uni-2}
\end{eqnarray}
Applying the H\"{o}lder and Young inequalities, and taking into account (\ref{interpol}), we obtain 
\begin{eqnarray}
-(u_1\nabla v,\nabla u)&\le&\|u_1\|_{L^4}\|\nabla v\|_{L^4}\|\nabla u\|\le C\|u_1\|_{L^4}\|\nabla v\|^{1/2}\|\nabla v\|^{1/2}_{H^1}\|\nabla u\|\nonumber\\
&\le&\delta(\|\nabla v\|^2_{H^1}+\|\nabla u\|^2)+C_\delta\|u_1\|^4_{L^4}\|\nabla v\|^2,\label{uni-3}\\
-(u\nabla v_2,\nabla u)&\le&\|u\|_{L^4}\|\nabla v_2\|_{L^4}\|\nabla u\|\le C\|u\|^{1/2}\|u\|_{H^1}^{1/2}\|\nabla v_2\|_{L^4}\|\nabla u\|\nonumber\\
&\le&\delta\|u\|^2_{H^1}+C_\delta\|\nabla v_2\|^4_{L^4}\|u\|^2,\label{uni-4}\\
(u,-\Delta v)&\le&\delta\|\Delta v\|^2+C_\delta\|u\|^2,\label{uni-5}\\
(fv,-\Delta v)&\le&\|f\|_{L^4}\|v\|_{L^4}\|\Delta v\|\le \delta\|v\|^2_{H^2}+C_\delta\|f\|^2_{L^4}\|v\|^2_{H^1}.\label{uni-6}
\end{eqnarray}
Replacing (\ref{uni-3})-(\ref{uni-6}) in (\ref{uni-2}), and using the fact that $\displaystyle\int_\Omega u(t)=0,\ \forall t>0;$ and
$$
\displaystyle\frac{d}{dt}\left(\displaystyle\int_\Omega v\right)+\int_\Omega v=\int_\Omega fv,
$$
hence
$$
\displaystyle\frac{d}{dt}\left(\int_\Omega v\right)^2+\left(\int_\Omega v\right)^2\le C\|f\|^2\|v\|^2,
$$

and by choosing $\delta$ small enough,  we have
\begin{eqnarray}
\frac{d}{dt}\left(\|u\|^2+\frac12\|v\|^2_{H^1}\right)+C(\|u\|^2_{H^1}+\|v\|^2_{H^2})
\le C(\|u_1\|^4_{L^4}\|\nabla v\|^2+(\|\nabla v_2\|^4_{L^4}+1)\|u\|^2+\|f\|^2_{L^4}\|v\|^2_{H^1}).\label{uni-7}
\end{eqnarray}
Therefore, from (\ref{uni-7}) and Gronwall lemma,  since $u_0=v_0=0$ and $(u_1,\nabla v_2)\in L^4(Q)\times L^4(Q)$, we obtain  $u=v=0$, and the uniqueness follows.

Thus, the proof of Theorem \ref{strong_solution} is finished.

\begin{remark}
Since $v\in\mathcal{Y}_v$, in particular $v\in L^\infty(Q)$. Thus, $v$ does not blow-up. Moreover, if initial data $u_0\in W^{5/4,8/3}(\Omega)$, we can obtain more regularity for $u$ and conclude that $u$
does not blow-up at finite time. Indeed, from (\ref{fij_4}) and (\ref{fij_5}) we deduce that
$u\in L^\infty(0,T; H^1(\Omega))\cap L^2(0,T;H^2(\Omega))\hookrightarrow L^q(Q)$, for $1\le q<\infty$. Then, taking into account that 
$\nabla u\in L^\infty(0,T;L^2(\Omega))\cap L^2(0,T;H^1(\Omega))\hookrightarrow L^4(Q)$, $\nabla v\in L^\infty(0,T;L^4(\Omega))\cap L^4(0,T;W^{1,4}(\Omega))\hookrightarrow L^8(Q)$,
and $\Delta v\in L^4(Q)$ we have
$
\nabla\cdot(u\nabla v)=u\Delta v+\nabla u\cdot\nabla v\in L^{8/3}(Q).
$
Thus,  Lemma \ref{feireisl} (for $p=8/3$) for  (\ref{pf-7})$_1$ allows us to conclude that 
$u\in L^\infty(0,T;W^{5/4,8/3}(\Omega))\cap L^{8/3}(0,T;W^{2,8/3}(\Omega))$, with $\partial_tu\in L^{8/3}(Q)$. In particular, we obtain
that $u\in L^\infty(Q)$.
\end{remark}
\begin{remark}
Cie\'{s}lak et al. \cite{cieslak}  studied system (\ref{eq1})-(\ref{eq3}) with $f\equiv0$. They proved the existence of classical solutions using the abstract theory for quasilinear parabolic 
systems developed by Amann \cite{amann}. This theory for classical solutions can be applied  here introducing a regularized problem related to (\ref{eq1})-(\ref{eq3}) by choosing 
a sequence of bilinear controls $\{f^\varepsilon\}_{\varepsilon>0}$, with $f^\varepsilon$ regular enough, such that $f^\varepsilon\rightarrow f$ in $L^4(Q)$, as $\varepsilon\rightarrow0$, 
and the corresponding regularization of the initial data. We would obtain a local unique classical solution $(u^\varepsilon,v^\varepsilon)$ 
of the regularized problem, but to obtain estimates for $u^\varepsilon$ and $v^\varepsilon$, independent of $\varepsilon$ and enough to pass to the limit, we must reproduce the same estimates 
that we have made using the Leray-Schauder fixed point theorem (see Lemma \ref{cot-1}, for the estimates, and Lemmas \ref{compact} and \ref{conti-operator}, for pass to the limit).
\end{remark}

\section{The Optimal Control Problem}
\label{sec:4}

In this section we establish the statement of the bilinear control problem under study. We suppose that 
$\mathcal{F}\subset L^4(Q_c):=L^4(0,T;L^4(\Omega_c))$ is a nonempty, closed and convex set, 
where $\Omega_c\subset\Omega$ is the control domain, and $\Omega_d\subset\Omega$ is the observability  domain. We consider data 
$u_0\in H^{1}(\Omega)$, $v_0\in W^{3/2,4}_{{\bf n}}(\Omega)$ with $u_0\ge 0$ and $v_0\ge 0$ in $\Omega$, 
 and  the function $f\in\mathcal{F}$ that describes the bilinear control acting 
on the  $v$-equation.

Now, we define the following constrained minimization problem
related to system (\ref{eq1})-(\ref{eq3}):
\begin{equation}\label{func}
\left\{
\begin{array}{l}
\mbox{Find  }(u,v,f)\in 
 \mathcal{M}
\mbox{ such that the functional }\\
J(u,v,f):=\dfrac{\alpha_u}{2}\displaystyle\int_0^T\int_{\Omega_d}|u(x,t)-u_d(x,t)|^2dxdt+
\frac{\alpha_v}{2}\int_0^T\int_{\Omega_d}|v(x,t)-v_d(x,t)|^2dxdt\vspace{0.2cm}\\
\hspace{2.2cm}+ \dfrac{N}{4}\displaystyle\int_0^T\int_{\Omega_c}|f(x,t)|^4dxdt\vspace{0.2cm}\\
\mbox{ is minimized, subject to $(u,v,f)$ satisfies the PDE system (\ref{eq1})-(\ref{eq3}) ,}
\end{array}
\right.
\end{equation}
where
\begin{equation}\label{M}
\mathcal{M}:=\mathcal{Y}_u\times \mathcal{Y}_v\times\mathcal{F}.
\end{equation}

Here $(u_d,\, v_d)\in L^2(Q_d)\times L^2(Q_d)$ represents  the desired states and the nonnegative real numbers $\alpha_u$, $\alpha_v$, and $N$ measure the cost of the states and control, respectively.
The set of admissible solutions of optimal control problem (\ref{func}) is defined by
\begin{equation}\label{adm}
\mathcal{S}_{ad}=\{s=(u,v,f)\in
\mathcal{M}\,:\, s \mbox{ is a strong solution of (\ref{eq1})-(\ref{eq3})}\}.
\end{equation}
The functional $J$ defined in (\ref{func}) describes the deviation of the cell density $u$ from a desired cell density $u_d$ and the deviation of the
chemical concentration $v$ from a desired chemical $v_d$, plus the cost of the control measured in the $L^4(\Omega)$-norm.

\subsection{Existence of global  Optimal Solution}

In this subsection we will prove the existence of a global  optimal solution of problem (\ref{func}). 
First we introduce the concept of optimal solution for  problem (\ref{func}).
\begin{definition}
An element $(\tilde{u},\tilde{v},\tilde{f})\in\mathcal{S}_{ad}$ will be called a global optimal solution of problem (\ref{func}) if
\begin{equation}\label{solution}
J(\tilde{u},\tilde{v},\tilde{f})=\min_{(u,v,f)\in\mathcal{S}_{ad}}J(u,v,f).
\end{equation}
\end{definition}

Thus, we have  the following result. 

\begin{theorem}\label{optimal_solution}
Let $u_0\in H^{1}(\Omega)$ and $v_0\in W^{3/2,4}_{{\bf n}}(\Omega)$ with $u_0\ge 0$ and $v_0\ge 0$ in $\Omega$. We assume that either $N>0$  or  $\mathcal{F}$ is bounded in $L^4(Q_c)$, then the optimal control problem (\ref{func})
has at least one global optimal solution $(\tilde{u},\tilde{v},\tilde{f})\in\mathcal{S}_{ad}$.
\end{theorem}
\begin{proof}
From Theorem \ref{strong_solution} we deduce that $\mathcal{S}_{ad}\neq\emptyset$. Let $\{s_ m\}_{m\in\mathbb{N}}=\{(u_m,v_m,f_m)\}_{m\in\mathbb{N}}\subset\mathcal{S}_{ad}$
a minimizing sequence of $J$, that is, $\displaystyle\lim_{m\rightarrow+\infty}J(s_m)=\inf_{s\in\mathcal{S}_{ad}}J(s)$. Then, by definition
of $\mathcal{S}_{ad}$,  for each $m\in\mathbb{N}$, 
$s_m$ satisfies the system
(\ref{eq1})-(\ref{eq3}), that is
\begin{eqnarray}
\partial_tu_m-\Delta u_m&=&\nabla\cdot(u_m\nabla v_m)\quad a.e.\ (t,x)\in Q,\label{op1}\\
\partial_tv_m-\Delta v_m+v_m&=&u_m+f_m v_m\quad a.e.\ (t,x)\in Q,\label{op2}\\
u_m(0)&=&u_0,\ v_m(0)=v_0\quad \hbox{in $\Omega$},\\
\dfrac{\partial u_m}{\partial{\bf n}}&=&0,\quad \dfrac{\partial v_m}{\partial{\bf n}}=0\ \mbox{ on }(0,T)\times\partial\Omega.\label{op22}
\end{eqnarray}
From the definition of $J$ and the assumption $N>0$  or  $\mathcal{F}$ is bounded in $L^4(Q_c)$, it follows that 
\begin{equation}\label{bound_F}
\{f_m\}_{m\in\mathbb{N}}\mbox{ is bounded in }L^{4}(Q_c).
\end{equation}
Also, from  (\ref{bound_solution}) there exists $C>0$, independent of $m$, such that
\begin{equation}\label{bound_u_v}
\|\partial_tu_m,\partial_tv_m\|_{L^2(Q)\times L^4(Q)}+\|u_m,v_m\|_{C(H^1\times W^{3/2,4}_{\bf n})}+\|u_m\|_{L^2(0,T;H^2(\Omega))}+\|v_m\|_{L^4(0,T;W^{2,4}(\Omega))}\le C.
\end{equation}

Therefore, from (\ref{bound_F}), (\ref{bound_u_v}), and taking into account that $\mathcal{F}$ is a closed convex subset of $L^4(Q_c)$
(hence is weakly closed in $L^4(Q_c)$), we deduce that there exists 
$\tilde{s}=(\tilde{u},\tilde{v},\tilde{f})\in 
 \mathcal{M}$
such that, for some subsequence of $\{s_m\}_{m\in\mathbb{N}}$, still denoted by  $\{s_m\}_{m\in\mathbb{N}}$, 
the following convergences hold, as $m\rightarrow+\infty$:
\begin{eqnarray}
u_m&\rightarrow&\tilde{u}\quad \mbox{weak in }L^{2}(0,T;H^2(\Omega))\mbox{ and weak * in }L^\infty(0,T;H^1(\Omega)),\label{c2}\\
v_m&\rightarrow&\tilde{v} \quad\mbox{weak in }L^4(0,T;W^{2,4}(\Omega))\mbox{ and weak * in }L^\infty(0,T;W^{3/2,4}_{\bf n}(\Omega)),\label{c3}\\
\partial_tu_m&\rightarrow&\partial_t\tilde{u} \quad\mbox{weak in } L^2(Q),\label{c4}\\
\partial_tv_m&\rightarrow&\partial_t\tilde{v} \quad\mbox{weak in } L^4(Q),\label{c5}\\
f_m&\rightarrow&\tilde{f} \quad \mbox{weak in } L^4(Q_c),\mbox{ and }\tilde{f}\in \mathcal{F}.\label{c6}
\end{eqnarray}

From (\ref{c2})-(\ref{c5}), the Aubin-Lions lemma (see \cite{lions}, Th\'eor\`eme 5.1, p.58)  and  using the Corollary 4 of \cite{simon} we have
\begin{eqnarray}
u_m&\rightarrow&\tilde{u}\quad\mbox{strongly in }C([0,T];L^2(\Omega))\cap L^2(0,T;H^1(\Omega)),\label{c7}\\
v_m&\rightarrow&\tilde{v}\quad\mbox{strongly in }C([0,T];L^4(\Omega))\cap L^4(0,T;W^{3/2,4}_{\bf n}(\Omega)).\label{c8}
\end{eqnarray}
In particular, 
since $\nabla \cdot \left(u_m \, \nabla v_m\right)= \nabla  u_m \cdot \nabla v_m + u_m \, \Delta v_m$
is bounded in $L^4(0,T;L^2(\Omega))$ and $f_m\, v_m$ is bounded in $L^4(Q_c)$, then one has the weak convergences
$$
\begin{array}{rcl}
\nabla\cdot(u_m\nabla v_m )&\rightarrow& \chi_1 \quad \hbox{ weak in $ L^4(0,T;L^2(\Omega))$},
\\
f_m v_m &\rightarrow&\chi_2 \quad \hbox{ weak in $ L^{4}(Q_c)$}.
\end{array}
$$
On the other hand, from (\ref{c2})-(\ref{c8}) one has:
$$
\begin{array}{rcl}
u_m\nabla v_m &\rightarrow&\tilde{u}\nabla\tilde{v} \quad \hbox{ weak in $L^2(0,T;L^4(\Omega))$},
\\
f_m v_m &\rightarrow& \tilde{f}\tilde{v} \quad \hbox{ weak in $L^4(0,T;L^2(\Omega_c))$}.
\end{array}
$$
Therefore, we can identify $\chi_1=\nabla\cdot(\tilde{u}\nabla\tilde{v})$ and $\chi_2=\tilde{f}\tilde{v}$ a.~e.~in $Q$, and thus:
\begin{eqnarray}
\nabla\cdot(u_m\nabla v_m )&\rightarrow& \nabla\cdot(\tilde{u}\nabla\tilde{v}) \quad \hbox{ weak in $ L^4(0,T;L^2(\Omega))$},\label{c8-0}\\
f_m v_m &\rightarrow& \tilde{f}\tilde{v} \quad \hbox{ weak in $ L^{4}(Q_c)$}.\label{c8-1}
\end{eqnarray}

Moreover, from (\ref{c7}) and (\ref{c8}) we have that $(u_m(0),v_m(0))$ converges to 
$(\tilde{u}(0),\tilde{v}(0))$ in $L^2(\Omega)\times L^4(\Omega)$, and since $u_m(0)=u_0$, $v_m(0)=v_0$, we deduce that $\tilde{u}(0)=u_0$ 
and $\tilde{v}(0)=v_0$, thus $\tilde{s}$ satisfies the initial conditions given in (\ref{eq2}). 
Therefore, considering the convergences (\ref{c2})-(\ref{c8-1}), we can pass to the limit in (\ref{op1})-(\ref{op22}) as $m$ goes to $+\infty$,
and we conclude that $\tilde{s}=(\tilde{u},\tilde{v},\tilde{f})$ is solution of the system pointwisely (\ref{eq1})-(\ref{eq3}), that is, $\tilde{s}\in\mathcal{S}_{ad}$. Therefore,
\begin{equation}\label{op20}
\lim_{m\rightarrow+\infty}J(s_m)=\inf_{s\in\mathcal{S}_{ad}}J(s)\le J(\tilde{s}).
\end{equation}
On the other hand, since $J$ is lower semicontinuous on $\mathcal{S}_{ad}$, we have $J(\tilde{s})\le \displaystyle\liminf_{m\rightarrow+\infty} J(s_m)$, which jointly to  
(\ref{op20}), implies (\ref{solution}).
\end{proof}

\subsection{Optimality System Related to Local Optimal Solutions}

In this subsection we will derive the first-order necessary optimality conditions for
a local optimal solution $(\tilde{u},\tilde{v},\tilde{f})$ of problem (\ref{func}), applying a Lagrange multipliers theorem.
We will base on a generic result given by Zowe et al  \cite{zowe} on existence of Lagrange multipliers in Banach spaces.
In order to introduce the concepts and results given in \cite{zowe} we consider the following optimization problem
\begin{equation}\label{abs1}
\min
J(s)\ \mbox{ subject to }
s \in \mathcal{S}=\{ s
\in \mathcal{M}\,:\, G( s
)\in \mathcal{N}\}.
\end{equation}
where $J:X\rightarrow\mathbb{R}$ is a functional, $G:X\rightarrow Y$ is an operator,
$X$ and $Y$ are Banach spaces,  $\mathcal{M}$ is a nonempty closed convex subset of $X$ and $\mathcal{N}$ is a nonempty closed convex cone in $Y$ with vertex at the origin.

For a subset $A$ of $X$ (or $Y$), $A^+$ denotes its polar cone, that is
$$
A^+=\{\rho\in X'\,:\, \langle \rho, a\rangle_{X'}\ge 0,\ \forall a\in A\}.
$$
\begin{definition}\label{local-solution}
We say that $\tilde{s}\in\mathcal{S}$ is a local optimal solution of problem (\ref{abs1}), if there exits $\varepsilon>0$ such that for all
$x\in\mathcal{S}$ satisfying $\|x-\tilde{x}\|_X\le\varepsilon$ one has that $J(\tilde{x})\le J(x)$.
\end{definition}

\begin{definition}\label{abs2}
Let $\tilde{s}\in\mathcal{S}$ be a local  optimal solution for problem (\ref{abs1}) 
with respect to the $X$-norm. 
Suppose that $J$ and $G$ are Fr\'echet differentiable in $\tilde{s}$, with derivatives
$J'(\tilde{s})$ and $G'(\tilde{s})$, respectively. Then, any $\lambda\in Y'$ is called a Lagrange multiplier for (\ref{abs1}) at the point $\tilde{s}$ if
\begin{equation}\label{abs3}
\left\{
\begin{array}{l}
\lambda\in \mathcal{N}^+,\\
\langle\lambda, G(\tilde{s})\rangle_{Y'}=0,\\
J'(\tilde{s})-\lambda\circ G'(\tilde{s})\in \mathcal{C}(\tilde{x})^+,
\end{array}
\right.
\end{equation}
where $\mathcal{C}(\tilde{s})=\{\theta( s- \tilde{s})\,:
\,  s \in \mathcal{M},\, \theta\ge0\}$ is the conical hull of $ \tilde{s}$ in $\mathcal{M}$. 
\end{definition}
\begin{definition}\label{abs4}
Let $\tilde{s}\in\mathcal{S}$ be a local optimal solution for problem (\ref{abs1}). We say that $\tilde{s}$ is a regular point if
\begin{equation}\label{abs5}
G'(\tilde{s})[\mathcal{C}(\tilde{s})]-\mathcal{N}(G(\tilde{s}))=Y,
\end{equation}
where $\mathcal{N}(G(\tilde{s}))
=\{(\theta(n-G(\tilde{s}))\,:\, n\in \mathcal{N},\, \theta\ge0\}$ 
is the conical hull of $G(\tilde{s})$ in $\mathcal{N}$.  
\end{definition}
\begin{theorem}(\cite{zowe}, Theorem 3.1)\label{abs6}
Let $\tilde{s}\in\mathcal{S}$ be a local  optimal solution for problem (\ref{abs1}).
Suppose that $J$ is a Fr\'echet differentiable function and $G$ is continuous Fr\'echet-differentiable.
If $\tilde{s}$ is a regular point, then the set of Lagrange multipliers for (\ref{abs1}) at $\tilde{s}$
is nonempty. 
\end{theorem}

Now, we will reformulate the optimal control problem (\ref{func}) in the abstract setting  (\ref{abs1}).
We consider the following Banach spaces
\begin{equation}\label{spaces_X_Y}
 X:=\mathcal{W}_u\times\mathcal{W}_v\times L^4(Q_c),\ 
Y:=L^2(Q)\times L^4(Q)\times H^1(\Omega)\times W^{3/2,4}_{\bf n}(\Omega),
\end{equation}
where
\begin{eqnarray}
\mathcal{W}_u&:=&\left\{u\in \mathcal{Y}_u\,:\,\dfrac{\partial u}{\partial{\bf n}}=0\ \mbox{ on }(0,T)\times\partial\Omega\right\},\label{space_state_u}\\
\mathcal{W}_v&:=&\left\{v\in \mathcal{Y}_v\,:\, \frac{\partial v}{\partial{\bf n}}=0\, \mbox{ on }(0,T)\times\partial\Omega\right\},\label{space_state_v}
\end{eqnarray} 
and the  operator $
{G}=(G_1,G_2,G_3,G_4):
X
\rightarrow
Y
$, where 
\begin{equation*}
\ G_1: X
\rightarrow L^2(Q),
\ G_2 : X
\rightarrow L^4(Q),
\ G_3: X
\rightarrow H^1(\Omega),
\ G_4: X
\rightarrow W^{3/2,4}_{\bf n}(\Omega)
\end{equation*}
are defined at each point $s=(u,v,f)\in
X$ by
\begin{equation}\label{restriction}
\left\{
\begin{array}{l}
G_1(s)=\partial_tu-\Delta u-\nabla\cdot(u\nabla v),\vspace{0.1cm}\\
G_2(s)=\partial_tv-\Delta v+v-u-fv,\vspace{0.1cm}\\
G_3(s)=u(0)-u_0,\vspace{0.1cm}\\
G_4(s)=v(0)-v_0.
\end{array}
\right.
\end{equation}

By taking $
\mathcal{M}
:=\mathcal{W}_u\times\mathcal{W}_v\times\mathcal{F}$ a closed convex subset of 
$
X$ and $\mathcal{N}=\{{\bf 0}\}$, the optimal control problem (\ref{func}) is reformulated  as follows
\begin{equation}\label{problem-1}
\min
J(s)\mbox{ subject to }
s \in \mathcal{S}_{ad}=\{s=(u,v,f)\in
\mathcal{M}
\,:\, {G}(s)={\bf 0}\}.
\end{equation}
Concerning to differentiability of the constraint operator ${G}$ and the functional $J$ we have the following results.
\begin{lemma}\label{derivative_J}
The functional $J:
X\rightarrow\mathbb{R}$ is 
Fr\'echet differentiable and the Fr\'echet derivative of $J$ in
$\tilde{s}=(\tilde{u},\tilde{v},\tilde{f})\in
X$ in the direction $r=(U,V,F)\in
X$ is given by
\begin{equation}\label{deri_J}
J'(\tilde{s})[r]=\alpha_u\int_0^T\int_{\Omega_d}(\tilde{u}-u_d)U\,dxdt+\alpha_v\int_0^T\int_{\Omega_d}(\tilde{v}-v_d)V\,dxdt+N\int_0^T\int_{\Omega_c}(\tilde{f})^3F\,dxdt.
\end{equation}
\end{lemma}
\begin{lemma}\label{derivative_R}
The operator ${G}:
 X\rightarrow
Y$ is 
continuous-Fr\'echet differentiable and the Fr\'echet derivative of ${G}$ in
$\tilde{s}=(\tilde{u},\tilde{v},\tilde{f})\in
X$, in the direction $r=(U,V,F)\in
 X$, is the linear operator
$${G}'(\tilde{s})[r]=(G_1'(\tilde{s})[r],G_2'(\tilde{s})[r],G_3'(\tilde{s})[r],G_4'(\tilde{s})[r])$$ defined by 
\begin{equation}\label{deri_R}
\left\{
\begin{array}{l}
G_1'(\tilde{s})[r]=\partial_tU-\Delta U-\nabla\cdot(U\nabla \tilde{v})-\nabla\cdot(\tilde{u}\nabla V),\vspace{0.1cm}\\
G_2'(\tilde{s})[r]=\partial_tV-\Delta V+V-U-\tilde{f}V-F\tilde{v},\vspace{0.1cm}\\
G_3'(\tilde{s})[r]=U(0),\vspace{0.1cm}\\
G_4'(\tilde{s})[r]=V(0).
\end{array}
\right.
\end{equation}
\end{lemma}
We wish to prove the existence of Lagrange multipliers,
which is guaranteed if a  local optimal solution of  problem (\ref{problem-1}) is a regular point  of operator  ${G}$
(see Theorem \ref{abs6}).
\begin{remark}\label{regular_point}
Since in the problem (\ref{problem-1}) $\mathcal{N}=\{{\bf 0}\}$, then $\mathcal{N}({G}(\tilde{s}))=\{{\bf 0}\}$. Thus, from Definition \ref{abs4}  we conclude that 
$\tilde{s}=(\tilde{u},\tilde{v},\tilde{f})\in\mathcal{S}_{ad}$
is a regular point  if for all $(g_u,g_v,U_0,V_0)\in
Y$ there exists $r=(U,V,F)\in \mathcal{W}_u\times\mathcal{W}_v\times\mathcal{C}(\tilde{f})$
such that
\begin{equation}\label{regular}
{G}'(\tilde{s})[r]=(g_u,g_v,U_0,V_0),
\end{equation}
where $\mathcal{C}(\tilde{f}):=\{\theta(f-\tilde{f})\,:\, \theta\ge0,\, f\in\mathcal{F}\}$ is the conical hull of  $\tilde{f}$ in $\mathcal{F}$. 
\end{remark}
\begin{lemma}\label{reg-2}
Let $\tilde{s}=(\tilde{u},\tilde{v},\tilde{f})\in\mathcal{S}_{ad}$, then $\tilde{s}$ is a regular point.
\end{lemma}
\begin{proof}
Fixed $(\tilde{u},\tilde{v},\tilde{f})\in \mathcal{S}_{ad}$, let $(g_u,g_v,U_0,V_0)\in
Y$. Since  $0\in\mathcal{C}(\tilde{f})$, it suffices to show the existence of $(U,V)\in\mathcal{Y}_u\times\mathcal{Y}_v$ such that
\begin{equation}\label{L-1}
\left\{\begin{array}{rcl}
\partial_tU-\Delta U-\nabla\cdot(U\nabla \tilde{v})-\nabla\cdot(\tilde{u}\nabla V)&=&g_u
 \quad \mbox{ in }Q,\\
\noalign{\vspace{-2ex}}\\
\partial_tV-\Delta V+V-U-\tilde{f}V&=&g_v \quad \mbox{ in }Q,\\
\noalign{\vspace{-2ex}}\\
U(0)&=&U_0,\ V(0)=V_0 \quad \mbox{ in }\Omega,\\
\noalign{\vspace{-2ex}}\\
\dfrac{\partial U}{\partial{\bf n}}&=&0,\ \dfrac{\partial V}{\partial {\bf n}}=0 \quad \mbox{ on }(0,T)\times\partial\Omega.
\end{array}
\right.
\end{equation} 
In order to prove the existence of solution of (\ref{L-1}), we can use the Leray-Schauder's fixed point argument over the operator $S:(\bar{U},\bar{V})\in \mathcal{X}_u \times \mathcal{X}_v \mapsto (U,V)\in  \mathcal{Y}_u \times \mathcal{Y}_v$ with $(U,V)$ the solution of the decoupled problem: 
$$
\left\{\begin{array}{rcl}
\partial_tU-\Delta U &=& \nabla\cdot(\tilde{u}\nabla {V}) + \nabla\cdot(\bar{U}\nabla \tilde{v})+g_u \quad\mbox{ in }Q,\\
\noalign{\vspace{-2ex}}\\
\partial_tV-\Delta V+V&=&  \bar{U}+\tilde{f}\bar{V} +g_v \quad \mbox{ in }Q,
\end{array}
\right.$$
endowed with the corresponding initial and boundary conditions. In fact, first we find $V$ and after $U$. 
%
%
Adapting Section \ref{existence}, we can prove that operator $S$ is well-defined  from $\mathcal{X}_u \times \mathcal{X}_v $ to $ \mathcal{Y}_u \times \mathcal{Y}_v $ and compact  from $\mathcal{X}_u \times \mathcal{X}_v $ to itself, due to the regularity 
$(g_u,g_v) \in L^2(Q)\times L^4(Q)$, $(U_0,V_0)\in H^1(\Omega)\times W^{3/2,4}_{\bf n}(\Omega)$, $(\tilde{u},\tilde{v})\in \mathcal{Y}_u \times \mathcal{Y}_v$ and $\tilde{f} \in L^4(Q)$.


It suffices to prove Lemma \ref{cot-1} but now defining $T_{\alpha}=\{
(U,V) \in \mathcal{Y}_u \times \mathcal{Y}_v: \, (U,V)=\alpha \, S(U,V) \, \mbox{ for some $\alpha\in [0,1]$}
\}$, which is very similar to the proof of Lemma \ref{cot-1}. In fact, now Steps 2 and 3 are easier and can be proved  jointly. 
%
%

With this objective, if $(U,V) \in T_{\alpha}$, then $(U,V)$ solves the coupled problem 
\begin{equation} \label{L-1aa}
\left\{\begin{array}{rcl}
\partial_tU-\Delta U - \nabla\cdot(\tilde{u}\nabla {V})&=&\alpha\nabla\cdot({U}\nabla \tilde{v})+ \alpha  g_u \quad\mbox{ in }Q,\\
\noalign{\vspace{-2ex}}\\
\partial_tV-\Delta V+V&=&  \alpha{U}+ \alpha\tilde{f}{V} + \alpha  g_v \quad \mbox{ in }Q,
\end{array}
\right.
\end{equation} 
endowed with the corresponding initial and boundary conditions. By testing by $(U,-\Delta V)$, one has
\begin{eqnarray}
\frac12\frac{d}{dt}\left(\|U\|^2+\|\nabla V\|^2\right)+\|\nabla U\|^2+\|\nabla V\|^2+\|\Delta V\|^2
&\le&
|(\tilde{u}\nabla V,\nabla U)|+
\alpha \, |(U\nabla \tilde{v},\nabla U)|
+ \alpha \,  |(U,\Delta V)|\nonumber\\
&&
+ \alpha \,  |(\tilde{f}V,\Delta V)|
+ \alpha \,  |(g_u,U)|
+ \alpha \,  |(g_v,\Delta V)|\label{L-2}.
\end{eqnarray}
Applying the H\"older and Young inequalities to the terms on the right side of (\ref{L-2}) and taking into account
(\ref{interpol}), we have
\begin{eqnarray}
|(\tilde{u}\nabla V,\nabla U)|&\le&\|\tilde{u}\|_{L^4}\|\nabla V\|_{L^4}\|\nabla U\|\le C_\delta\|\tilde{u}\|^2_{L^4}\|\nabla V\|\|\nabla V\|_{H^1}+\delta\|\nabla U\|^2\nonumber\\
&\le&
\delta(\|\nabla V\|^2_{H^1}+\|\nabla U\|^2)
+ C_\delta\|\tilde{u}\|^4_{L^4}\|\nabla V\|^2\label{L-4},\\
\noalign{\vspace{-2ex}}
 \nonumber \\
\alpha \,
|(U\nabla \tilde{v},\nabla U)|
&\le&
\alpha \, 
\|U\|_{L^4}\|\nabla \tilde{v}\|_{L^4}\|\nabla U\|
\le 
C \, \alpha \, \|U\|^{1/2}\|\nabla \tilde{v}\|_{L^4}\|U\|^{3/2}_{H^1}\nonumber\\
&\le&\delta\|U\|^2_{H^1}+C_\delta 
\,
\alpha^2  \,
\|\nabla \tilde{v}\|^4_{L^4}\|U\|^2,\label{L-3}\\
\noalign{\vspace{-2ex}} \nonumber\\
\alpha \, |(g_u,U)|
&\le&
\alpha \, \left(\|g_u\|^2+\|U\|^2 \right),
\label{L-5}\\
\noalign{\vspace{-2ex}} \nonumber\\
\alpha \, |(U,\Delta V)|
&\le&
\delta\|\Delta V\|^2+ 
C_\delta \, \alpha^2 \, \|U\|^2,\label{L-6}\\
\noalign{\vspace{-2ex}} \nonumber\\
\alpha \, |(\tilde{f}V,\Delta V)|
&\le&
\alpha \, 
\|\tilde{f}\|_{L^4}\|V\|_{L^4}\|\Delta V\|
\le
\delta\|\Delta V\|^2
 +
 C_\delta \, \alpha^2 \, \|\tilde{f}\|^2_{L^4}\|V\|^2_{H^1}
 ,\label{L-7}\\
\noalign{\vspace{-2ex}} \nonumber\\
\alpha \,  |(g_v,\Delta V)|
&\le&
 \delta\|\Delta V\|^2
  + C_\delta \alpha^2 \, \|g_v\|^2
 .\label{L-8}
\end{eqnarray}
On the other hand, testing (\ref{L-1aa})$_2$ by  $V$ one  has
\begin{eqnarray}\label{L-1a}
\frac12\frac{d}{dt}\|V\|^2+\| V\|_{H^1}^2
&\le&
\alpha \, |(U,V)|
+ \alpha \, |(\tilde{f}V,V)| 
+ \alpha \,  \vert (g_v,V) \vert
\nonumber\\
&\le&\delta \, \left(  \| V\|^2 + \Vert  V \Vert^2_{H^1} \right)
+
C_\delta \, \alpha^2 \, \|U\|^2
\nonumber\\
& + & C_\delta \alpha^2 \|\tilde{f}\|^2_{L^4}\|V\|^2 
+ C_\delta \, \alpha^2 \, \Vert g_v \Vert^2.
\end{eqnarray}
Summing the inequalities (\ref{L-2}) and  (\ref{L-1a}), and 
then adding $\|U\|^2$ to both sides of the obtained inequality 
and considering (\ref{L-3})-(\ref{L-8}), taking $\delta$ small enough and any $\alpha \in [0,1]$, we have
\begin{eqnarray}
\frac{d}{dt}(\|U\|^2+\|V\|^2_{H^1})+C(\|U\|^2_{H^1}+\|V\|^2_{H^2})
&\le&C(1+\|\nabla \tilde{v}\|^4_{L^4})\|U\|^2
+C(\|g_u\|^2+\|g_v\|^2)\nonumber\\
&&+C(\|\tilde{u}\|^4_{L^4}+\|\tilde{f}\|^2_{L^4} ) \|V\|^2_{H^1}.\label{L-10}
\end{eqnarray}




By applying the Gronwall Lemma in (\ref{L-10}) we conclude that there exists a  positive constant $C_0$ that depends on
$T,\|U_0\|,\|V_0\|_{H^1},\|\tilde{u}\|_{L^4(Q)},\|\nabla\tilde{v}\|_{L^4(Q)},\|\tilde{f}\|_{L^2(L^4)},\|g_u\|_{L^2(Q)}$ 
and $\|g_v\|_{L^2(Q)},$ such that 
\begin{equation}\label{L-11}
\|U,V\|_{L^\infty(L^2\times H^1)\cap L^2(H^1\times H^2)}\le C_0.
\end{equation}

Now, 
following  Step 4  in the proof of Lemma \ref{cot-1}, 
we obtain that $V$ is bounded in $\mathcal{Y}_v$, because  the following estimate holds
\begin{equation}\label{L-110}
\|V\|_{L^4(W^{2,4}(\Omega))}+\|\partial_tV\|_{L^4(Q)}+\|V\|_{C(W^{3/2,4}_{\bf n})}
\le C_1(C_0, \|V_0\|_{W^{3/2,4}_{\bf n}},\|g_v\|_{L^4(Q)},\|\tilde{f}\|_{L^4(Q)}).
\end{equation}

Now, we follow Step 5  in the proof of Lemma \ref{cot-1} with small modifications. By testing (\ref{L-1aa})$_1$ by $-\Delta U$, using the H\"older and Young inequalities, and considering the interpolation inequality (\ref{interpol}),  we  obtain
\begin{eqnarray}
\frac12\frac{d}{dt}\|\nabla U\|^2+\|\Delta U\|^2
&\le& C \, \alpha^2  \, (\|g_u\|^2+\|U\|^2_{L^4}\|\Delta\tilde{v}\|^2_{L^4}+\|\nabla U\|^2\|\nabla\tilde{v}\|^4_{L^4})\nonumber\\
&&+C(\|\tilde{u}\|^2_{L^4}\|\Delta V\|^2_{L^4}+\|\nabla\tilde{u}\|^2_{L^4}\|\nabla V\|^2_{L^4})+\delta(\|\Delta U\|^2+\|U\|^2_{H^2}).\label{L-16}
\end{eqnarray}
On the other hand, from (\ref{L-1})$_1$ we deduce $\displaystyle\frac{d}{dt}\left(\int_\Omega U\right)=\alpha \, \int_\Omega g_u$, which implies
\begin{equation}
\frac12\frac{d}{dt}\left(\int_\Omega U\right)^2
=
\alpha \, \left(\int_\Omega g_u\right)\left(\int_\Omega U\right)
\le 
C_\delta \, \alpha^2 \,
\left(\int_\Omega g_u\right)^2+\delta\left(\int_\Omega U\right)^2,\label{L-17-0}
\end{equation}
and
\begin{equation}
\left|\int_\Omega U(t)\right|^2=\left|\int_\Omega U_0+\alpha \, \int_0^t\int_\Omega g_u\right|^2\le C.\label{L-17}
\end{equation}
Summing inequalities (\ref{L-16})-(\ref{L-17}), taking $\delta$ small enough and $\alpha \in [0,1]$,  accounting (\ref{equi}), (\ref{norma-1}) and (\ref{L-110}), we can obtain the estimate
$\|U\|_{L^\infty(H^1)\cap L^2(H^2)}\le C$. Finally, the estimate $\|\partial_t U\|_{L^2(L^2)}\le C$ is deduced as in the proof of Lemma \ref{cot-1}.

 Therefore, we can deduce the existence of solution for (\ref{L-1}) from Leray-Schauder fixed-point theorem.

The uniqueness of $(U,V)$ follows directly from the regularity of $(U,V)$ and the linearity of system (\ref{L-1}).

Thus, we conclude the proof.
\end{proof}
Now we show the existence of Lagrange multipliers.

\begin{theorem}\label{lagrange}
Let $\tilde{s}=(\tilde{u},\tilde{v},\tilde{f})\in\mathcal{S}_{ad}$ be a local optimal solution for the control problem
(\ref{problem-1}). Then, there exist Lagrange multipliers $(\lambda,\eta, \xi,\varphi)\in L^2(Q)\times L^{4/3}(Q)\times (H^1(\Omega))'\times (W^{3/2,4}_{\bf n}(\Omega))'$ such that for all
$(U,V,F)\in\mathcal{W}_u\times\mathcal{W}_v\times \mathcal{C}(\tilde{f})$ one has 
\begin{eqnarray}\label{reg-4}
&&\alpha_u\int_0^T\int_{\Omega_d}(\tilde{u}-u_d)U\,dxdt+\alpha_v\int_0^T\int_{\Omega_d}(\tilde{v}-v_d)V\,dxdt+N\int_0^T\int_{\Omega_c}(\tilde{f})^3F\,dxdt\nonumber\\
&&-\int_0^T\int_\Omega\bigg(\partial_tU-\Delta U-\nabla\cdot(U\nabla\tilde{v})-\nabla\cdot(\tilde{u}\nabla V)\bigg)\lambda\,dxdt\nonumber\\
&&-\int_0^T\int_\Omega\bigg(\partial_tV-\Delta V+V-U-\tilde{f}V\bigg)\eta\,dxdt\nonumber\\
&&-\int_\Omega U(0)\xi\,dx-\int_\Omega V(0)\varphi\,dx+\int_0^T\int_{\Omega_c}F\tilde{v}\eta\,dxdt\ge0.
\end{eqnarray}
\end{theorem}
\begin{proof}

From Lemma \ref{reg-2}, $\tilde{s}\in\mathcal{S}_{ad}$ is a regular point, then from Theorem \ref{abs6} there exist Lagrange multipliers
$(\lambda,\eta, \xi,\varphi)\in L^2(Q)\times L^{4/3}(Q)\times (H^1(\Omega))'\times (W^{3/2,4}_{\bf n}(\Omega))'$ such that
\begin{equation}\label{reg-5}
J'(\tilde{s})[r]-\langle R'_1(\tilde{s})[r],\lambda\rangle-\langle R_2'(\tilde{s})[r],\eta\rangle
-\langle R'_3(\tilde{s})[r],\xi\rangle-\langle R'_4(\tilde{s})[r],\varphi\rangle\ge 0,
\end{equation}
for all  $r=(U,V,F)\in\mathcal{W}_u\times\mathcal{W}_v\times \mathcal{C}(\tilde{f})$.
Thus, the proof follows from (\ref{deri_J})-(\ref{deri_R}).
\end{proof}
From Theorem, \ref{lagrange} we derive an optimality system for which we consider the following spaces 
\begin{equation}\label{reg-51}
\mathcal{W}_{u_0}:=\{u\in\mathcal{W}_u\,:\, u(0)=0\},\ \mathcal{W}_{v_0}:=\{v\in\mathcal{W}_v\,:\, v(0)=0\}.
\end{equation}

\begin{corol}\label{opti_system}
Let $\tilde{s}=(\tilde{u},\tilde{v},\tilde{f})$ be a local optimal solution for the optimal control problem (\ref{problem-1}). Then the Lagrange multiplier
$(\lambda,\eta)\in L^2(Q)\times L^{4/3}(Q)$, provided by Theorem \ref{lagrange}, satisfies the system
\begin{eqnarray}
&&\int_0^T\int_\Omega\bigg(\partial_tU-\Delta U-\nabla\cdot(U\nabla\tilde{v})\bigg)\lambda\,dxdt-\int_0^T\int_\Omega U\eta\,dxdt\nonumber\\
&&=\alpha_u\int_0^T\int_{\Omega_d}(\tilde{u}-u_d)U\,dxdt\quad \forall U\in\mathcal{W}_{u_0},\label{reg-6}\\
&&\int_0^T\int_\Omega\bigg(\partial_tV-\Delta V+V\bigg)\eta\,dxdt-\int_0^T\int_{\Omega_c} \tilde{f}V\eta\, dxdt-\int_0^T\int_\Omega\nabla\cdot(\tilde{u}\nabla V)\lambda\,dxdt\nonumber\\
&&=\alpha_v\int_0^T\int_{\Omega_d}(\tilde{v}-v_d)V\,dxdt\quad \forall V\in\mathcal{W}_{v_0},\label{reg-7}
\end{eqnarray}
which corresponds to the concept of very weak solution of the linear system 
\begin{equation}\label{reg-8}
\left\{
\begin{array}{rcl}
\partial_t\lambda+\Delta\lambda-\nabla\lambda\cdot\nabla\tilde{v}+\eta&=&-\alpha_u(\tilde{u}-u_d)\chi_{_{\Omega_d}}\ \mbox{ in }Q,\\
\partial_t\eta+\Delta \eta +\nabla\cdot(\tilde{u}\nabla\lambda)-\eta+\tilde{f}\eta\chi_{_{\Omega_c}}&=&-\alpha_v(\tilde{v}-v_d)\chi_{_{\Omega_d}}\ \mbox{ in }Q,\\
\lambda(T)&=&0,\ \eta(T)=0\ \mbox{ in }\Omega,\\
\dfrac{\partial\lambda}{\partial{\bf n}}&=&0,\ \dfrac{\partial\eta}{\partial{\bf n}}=0\ \mbox{ on }(0,T)\times\partial\Omega,
\end{array}
\right.
\end{equation}
and the optimality condition  
\begin{equation}\label{reg-9}
\int_0^T\int_{\Omega_c}(N(\tilde{f})^3+\tilde{v}\eta)(f-\tilde{f})\,dxdt\ge 0,\ \forall f\in\mathcal{F}.
\end{equation}
\end{corol}
\begin{proof}
From (\ref{reg-4}), taking $(V,F)=(0,0)$, and taking into account that $\mathcal{W}_{u_0}$ is a vectorial space, we have (\ref{reg-6}). Similarly, taking $(U,F)=(0,0)$ in (\ref{reg-4}), and considering
that $\mathcal{W}_{v_0}$ is a vectorial space we obtain (\ref{reg-7}).
Finally, taking $(U,V)=(0,0)$ in (\ref{reg-4}) we have

$$
N\int_0^T\int_{\Omega_c}(\tilde{f})^3F\,dxdt+\int_0^T\int_{\Omega_c}\tilde{v}\eta F\,dxdt\ge0,\ \forall F\in \mathcal{C}(\tilde{f}).
$$
Therefore, choosing $F=f-\tilde{f}\in \mathcal{C}(\tilde{f})$ for all $f\in\mathcal{F}$ in the last inequality, we obtain (\ref{reg-9}).
\end{proof}

In the following result we show that the Lagrange multiplier $(\lambda,\eta)$, provided by Theorem \ref{lagrange}, has some extra regularity. 

\begin{theorem}\label{adjoint}
Under of conditions of Theorem \ref{lagrange}, system (\ref{reg-8}) has a unique  strong solution  
$(\lambda,\eta)$ such that
\begin{eqnarray}
&&\lambda \in L^\infty(0,T;H^1(\Omega))\cap L^2(0,T;H^2(\Omega)),\ \partial_t\lambda\in L^2(Q),\label{adj-1}\\
&&\eta\in L^\infty(0,T;W^{2-2/p,p}(\Omega))\cap L^p(0,T;W^{2,p}(\Omega)),\ \partial_t\eta\in L^p(Q), \quad \hbox{for any $p<2$.}
\label{adj-2}
\end{eqnarray}

\end{theorem}
\begin{proof}
Let $s=T-t,$ with $t\in(0,T)$ and $\tilde{\lambda}(s)=\lambda(t)$, $\tilde{\eta}(s)=\eta(t)$. Then system (\ref{reg-8}) is equivalent to
\begin{equation}\label{adj}
\left\{
\begin{array}{rcl}
\partial_s\tilde{\lambda}-\Delta\tilde{\lambda}+\nabla\tilde{\lambda}\cdot\nabla\tilde{v}-\tilde{\eta}&=&\alpha_u(\tilde{u}-u_d)\chi_{_{\Omega_d}}\ \mbox{ in }Q,\\
\partial_s\tilde{\eta}-\Delta\tilde{\eta}-\nabla\cdot(\tilde{u}\nabla\tilde{\lambda})+\tilde{\eta}-\tilde{f}\tilde{\eta}\chi_{_{\Omega_c}}&=&\alpha_v(\tilde{v}-v_d)\chi_{_{\Omega_d}}\ \mbox{ in }Q,\\
\tilde{\lambda}(0)&=&0,\ \tilde{\eta}(0)=0\ \mbox{ in }\Omega,\\
\dfrac{\partial\tilde{\lambda}}{\partial{\bf n}}&=&0,\ \dfrac{\partial\tilde{\eta}}{\partial{\bf n}}=0\ \mbox{ on }(0,T)\times\Omega.
\end{array}
\right.
\end{equation}
Following an analogous reasoning  as in the proof of Lemma \ref{reg-2}, we can obtain the  energy inequality
\begin{eqnarray}\label{ad-661}
\frac{d}{ds}(\|\tilde\lambda\|^2_{H^1}+\|\tilde\eta\|^2)+C(\|\tilde\lambda\|^2_{H^2}+\|\tilde\eta\|^2_{H^1})
&\le& C(\|\tilde\eta\|^2+\|\tilde\lambda\|^2)+C(\|\tilde{u}-u_d\|^2+\|\tilde{v}-v_d\|^2)\nonumber\\
&&+C\|\tilde\lambda\|^2_{H^1}\|\nabla \tilde{v}\|^4_{L^4}+C\|\tilde{u}\|^4_{L^4}\|\nabla\tilde\lambda\|^2+C\|\tilde{f}\|^4_{L^4}\|\tilde\eta\|^2\nonumber\\
&\le& C(1+\|\tilde{f}\|_{L^4}^4)\|\tilde\eta\|^2+C(\|\tilde{u}-u_d\|^2+\|\tilde{v}-v_d\|^2)\nonumber\\
&&+C(1+\|\tilde{u}\|^4_{L^4}+\|\nabla\tilde{v}\|^4_{L^4})\|\tilde\lambda\|^2_{H^1}.
\end{eqnarray}
Thus,  we deduce that
$$
\left\{
\begin{array}{l}
\tilde\lambda\in L^\infty(0,T;H^1(\Omega))\cap L^2(0,T;H^2(\Omega)),\\
\tilde\eta \in L^\infty(0,T;L^2(\Omega))\cap L^2(0,T;H^1(\Omega)),
\end{array}
\right.
$$
hence in particular (\ref{adj-1}) holds. 

Now, since $\tilde{f}\in L^4(Q_c)$ and $\tilde\eta\in L^\infty(0,T;L^2(\Omega))\cap L^2(0,T;H^1(\Omega))\hookrightarrow L^4(Q)$ we have
\begin{equation}\label{adj-3}
\tilde{f}\tilde\eta \in L^2(Q).
\end{equation}
Also, taking into account that 
$\tilde{u}\in\mathcal{W}_u$, where $\mathcal{W}_u$ is defined in (\ref{space_state_u}), and $\tilde{\lambda}\in L^\infty(0,T;H^1(\Omega))\cap L^2(0,T;H^2(\Omega))$, we
obtain
\begin{equation}\label{adj-4}
\nabla\cdot(\tilde{u}\nabla\tilde\lambda)=\tilde{u}\Delta\tilde\lambda+\nabla\tilde{u}\cdot\nabla\tilde\lambda\in L^p(Q)\quad \forall p<2.
\end{equation}
Therefore, from (\ref{adj})$_2$, (\ref{adj-3}), (\ref{adj-4}) and Lemma \ref{feireisl} we conclude (\ref{adj-2}).
\end{proof}
\begin{corol}(Optimality System)\label{optimality_system}
Let $\tilde{s}=(\tilde{u},\tilde{v},\tilde{f})\in\mathcal{S}_{ad}$ be a local optimal solution for the control problem
(\ref{problem-1}). Then, the Lagrange multiplier $(\lambda,\eta)$ satisfies the regularity \eqref{adj-1} and \eqref{adj-2} and the following optimality system  
\begin{equation}\label{optimality}
\left\{
\begin{array}{rcl}
\partial_t\lambda+\Delta\lambda-\nabla\lambda\cdot\nabla\tilde{v}+\eta&=&-\alpha_u(\tilde{u}-u_d)\chi_{_{\Omega_d}}\ \mbox{ a.e. }(t,x)\in Q,\\
\partial_t\eta+\Delta \eta +\nabla\cdot(\tilde{u}\nabla\lambda)-\eta+\tilde{f}\eta\chi_{_{\Omega_c}}&=&-\alpha_v(\tilde{v}-v_d)\chi_{_{\Omega_d}}\ \mbox{ a.e. }(t,x)\in Q,\\
\lambda(T)&=&0,\ \eta(T)=0\ \mbox{ in }\Omega,\\
\dfrac{\partial\lambda}{\partial{\bf n}}&=&0,\ \dfrac{\partial\eta}{\partial{\bf n}}=0\ \mbox{ on }(0,T)\times\partial\Omega,\\
\displaystyle\int_0^T\int_{\Omega_c}(N(\tilde{f})^3+\tilde{v}\eta)(f-\tilde{f})\,dxdt&\ge& 0,\ \forall f\in\mathcal{F}.
\end{array}
\right.
\end{equation}
\end{corol}
\begin{remark}
If $\mathcal{F}\equiv L^4(Q_c)$, that is, there is no convexity constraint on the control, then,  (\ref{optimality})$_5$ becomes
$$
N(\tilde{f})^3\chi_{_{\Omega_c}}+\tilde{v}\eta\chi_{_{\Omega_c}}=0.
$$
Thus, the control $\tilde{f}$ is given by
\begin{equation}\label{control}
\tilde{f}=\left(-\frac1N\tilde{v}\eta\right)^{1/3}\chi_{_{\Omega_c}}.
\end{equation}
\end{remark}

\begin{remark}
All the results obtained in this work hold when the control $f$ belong to $L^q(Q)$, for $q>2$.  Indeed,
we obtain the existence of pointwise strong solutions $(u,v)$ of (\ref{eq1})-(\ref{eq3}), where
the regularity for $u$ remains fixed, that is, $u\in L^\infty(0,T;H^1(\Omega))\cap L^2(0,T;H^2(\Omega))$ with $\partial_tu\in L^2(Q)$, and
$v\in L^\infty(0,T;W^{2-2/q,q}(\Omega))\cap L^q(0,T;W^{2,q}(\Omega))$ with $\partial_tv\in L^q(Q)$.
We fix $q=4$ only for simplicity in the notation.
\end{remark}

\subsection*{Acknowledgments}

F. Guill\'en-Gonz\'alez and M.A. Rodr\'iguez-Bellido have been supported by MINECO grant MTM2015-69875-P (Ministerio de Econom\'ia y Competitividad, Spain). E. Mallea-Zepeda 
has been supported by Proyecto UTA-Mayor 4740-18 (Universidad de Tarapac\'a, Chile). Also, E. Mallea-Zepeda expresses his gratitude to Instituto de Matem\'aticas 
Universidad de Sevilla and Dpto. de Ecuaciones Diferenciales y An\'alisis Num\'erico of Universidad de Sevilla for their hospitality during his research stay in both centers.



\begin{thebibliography}{99}

\bibitem{alekseev} G.V. Alekseev,  Mixed boundary value problems for steady-state magnetohydrodynamic equations of viscous incompressible fluids.
{\it Comp. Math. Math. Phys.} {\bf 56}, no. 8 (2016) 1426-1439.

\bibitem{amann} H. Amann,  Nonhomogeneous linear and quasilinear elliptic and parabolic boundary value problems. {\it Function spaces, differential operators
and nonlinear analysis. Teubner-Texte Math.} {\bf 133} (1993) 9-126. Teubner, Stuttgart.


\bibitem{borzi_park} A. Borz\`{i}, E.-J. Park and  M. Vallejos Lass, Multigrid optimization methods for the optimal control of convection diffusion problems with bilinear control.
{\it J. Optim. Theory Appl.} {\bf 168} (2016) 510-533.

\bibitem{chaves_guerrero_1} F.W. Chaves-Silva and S. Guerrero,  A uniform controllability for the Keller-Segel system.
{\it Asymptot. Anal.} {\bf 92}, no. 3-4 (2015) 318-338.

\bibitem{chaves_guerrero_2} F.W. Chaves-Silva and S. Guerrero, A controllability result for a chemotaxis-fluid model.
{\it J. Diff. Equations.}  {\bf 262}, no. 9 (2017) 4863-4905.

\bibitem{cieslak} T. Cie\'{s}lak, P. Lauren\c{c}ot and C. Morales-Rodrigo, Global existence and convergence to steady states in a chemorepulsion system.
{\it Parabolic and Navier-Stokes equations. Part 1. Banach Center Publ., 81. Banach Center Publ., 81, Part 1, Polish Acad. Sci. Inst. Math., Warsaw.} (2008) 105-117.

\bibitem{casas_kunisch} E. Casas and K. Kunisch, Stabilization by space controls for a class of semilinear parabolic equations. 
{\it SIAM J. Control Optim.} {\bf 55}, no. 1 (2017) 512-532.

\bibitem{dearaujo} A.L.A. De Araujo and P.M.D. De Magalh\~aes, Existence of solutions and optimal control for a model of tissue invasion by solid tumours.
{\it J. Math. Anal. Appl.} {\bf 421} (2015) 842-877.

 
\bibitem{feireisl} E. Feireisl and A. Novotn\'y, Singular limits in thermodynamics of viscous fluids. Advances in Mathematical Fluid Mechanics. Birkh\"auser Verlag, Basel (2009).

\bibitem{fister_mccarthy} K.R. Fister and C.M. Mccarthy, Optimal control of a chemotaxis system.
{\it Quart. Appl. Math.} {\bf 61}, no. 2, (2003) 193-211.

\bibitem{karl_whachsmuth} V. Karl and D. Wachsmuth, An augmented Lagrange method for elliptic state constrained optimal control problems.
{\it Comp. Optim. Appl.} {\bf 69} (2018) 857-880.

\bibitem{keller-segel} E.F. Keller  and L.A. Segel, Initiation of slime mold aggregation viewed as an instability.
{\it J. Theor. Biol.} {\bf 26} (1970) 399-415.

\bibitem{kiem_rosch} B.T.  Kien, A. R\"osch and D. Wachsmuth, Pontyagin's principle for optimal control problem governed by 3D Navier-Stokes equations.
{\it J. Optim. Theory Appl.} {\bf 173} (2017) 30-55.

\bibitem{kroner_vexler} A.  Kr\"oner and B. Vexler,  A priori error estimates for elliptic optimal control problems with bilinear state equation. 
{\it J. Comp. Appl. Mech.} {\bf 230} (2009) 781-802.


\bibitem{kunisch_trautman} K. Kunisch, P. Trautmann and B. Vexler, Optimal control of the undamped linear wave equation with measure valued controls.
{\it SIAM J. Control Optim.} {\bf 54}, no. 3 (2016) 1212-1244.

\bibitem{lions} J.L. Lions,  Quelques m\'etodes de r\'esolution des probl\`emes aux limites non lin\'eares. Dunod, Paris (1969).


\bibitem{elva_elder} E. Mallea-Zepeda, E. Ortega-Torres and E.J. Villamizar-Roa, A boundary control problem for micropolar fluids.
{\it J. Optim. Theory Appl.} {\bf 169}, no. 2 (2016) 349-369.

\bibitem{necas} L. Necas, Les m\'ethodes directes en th\'eorie des equations elliptiques. Editeurs Academia, Prague (1967).

\bibitem{rodriguez_rueda} M.A. Rodr\'iguez-Bellido, D.A. Rueda-G\'omez and E.J.  Villamizar-Roa, On a distributed control problem for a coupled
chemotaxis-fluid model.
{\it Discrete Cotin. Dyn. Syst. B.} {\bf 23}, no. 2 (2018) 557-517.

\bibitem{diego_elder} D.A. Rueda-G\'omez and  E.J. Villamizar-Roa, On the Rayleigh-B\'enard-Marangoni system and a related optimal control problem. 
{\it Comp. Math. Appl.} {\bf 74}, no. 12 (2017) 2969-2991.

\bibitem{ryu_yagi} S.-U. Ryu and A. Yagi, Optimal control of Keller-Segel equations.
{\it J. Math. Anal. Appl.} {\bf 256}, no. 1 (2001) 45-66.

\bibitem{ryu} S.-U. Ryu, Boundary control of chemotaxis reaction diffusion system.
{\it Honam Math. J.} {\bf 30}, no. 3 (2008) 469-478.

\bibitem{simon} J. Simon, Compact sets in the space $L^p(0,T;B)$. {\it Ann. Mat. Pura Appl.} {\bf 146} (1987) 65-96.

\bibitem{tachim} T. Tachim Medjo, Optimal control of the primitive equations of the ocean with state constraints. {\it Nonlinear Analysis} {\bf 73} (2010) 634-649.

\bibitem{tao} Y. Tao, Global dynamics in a higher-dimensional repulsion chemotaxis model with nonlinear sensivity.
{\it Discrete Cotin. Dyn. Syst. B.} {\bf 18}, no. 10 (2013) 2705-2722. 

\bibitem{triebel} H. Triebel, Interpolation theory, function spaces, differential operators. 
VEB Deutscher Verlag de Wissenschaften, Berlin (1978).

\bibitem{vallejos_borzi} M. Vallejos and A. Borz\`{i}, Multigrid optimization methods for linear and bilinear elliptic optimal control problems. 
{\it Computing}  {\bf 82}, no. 2 (2008) 31-52.

\bibitem{wang-1} G. Wang, Optimal control of 3-dimensional Navier-Stokes equations with state constraints. 
{\it SIAM J. Control Optim.}  {\bf 41} (2002) 583-606.

\bibitem{zheng} J. Zhen and Y. Wang, Optimal control problem for Cahn-Hilliard equations with state constraints. 
{\it J. Dyn. Control Syst.}  {\bf 21} (2015) 257-272.

\bibitem{zowe} J. Zowe and S. Kurcyusz, Regularity and stability for the mathematical programming problem in Banach spaces.
{\it Appl. Math. Optim.}  {\bf 5} (1979) 49-62.

\end{thebibliography}
\end{document}